%% file: main.tex
\documentclass[12pt]{article}
\usepackage[parfill]{parskip} 
\usepackage{microtype}
\usepackage{amssymb}
\usepackage{amsthm}
\usepackage{mathtools}
\usepackage[hidelinks]{hyperref}
\usepackage{geometry} 
\usepackage{xcolor}

\theoremstyle{plain}
\newtheorem{theorem}{Theorem}[section]

\newtheorem{lemma}[theorem]{Lemma}
\newtheorem{proposition}[theorem]{Proposition}
    
\theoremstyle{definition}
\newtheorem{definition}[theorem]{Definition}
\newtheorem{example}[theorem]{Example}
\newtheorem{remark}[theorem]{Remark}

\numberwithin{equation}{section}

\DeclareMathOperator{\tr}{tr}
\let\div\relax
\DeclareMathOperator{\div}{div}
\newcommand{\Ric}{\mathrm{Ric}}
\newcommand{\Riem}{\mathrm{Rm}}
\newcommand{\Scal}{\mathrm{Scal}}
\newcommand{\Sc}{\mathcal S^g}
\DeclareMathOperator{\grad}{grad}

\newcommand{\ScalOp}[1]{\langle\kern-1.9pt\langle \Scal^{g},#1 \rangle\kern-1.9pt\rangle}
\newcommand{\ScOp}[1]{\int_M #1 \Sc d\mu_h}

\title{A generalization of the ADM mass for asymptotically Euclidean manifolds of weak regularity}
\author{Stig Lundgren\footnote{Nagoya University, Nagoya, Japan}\,\, \& Benjamin Meco\footnote{Uppsala University, Uppsala, Sweden}}
\date{}

\begin{document}
\maketitle
\begin{abstract}
    We propose a new definition of the ADM mass for asymptotically Euclidean manifolds inspired by the definition of mass for weakly regular asymptotically hyperbolic manifolds by Gicquaud and Sakovich. This version of the mass allows one to work with metrics of local Sobolev regularity $ W^{1,2}_\text{loc} \cap L^\infty $ and we show, under suitable asymptotic assumptions, that the mass is finite, invariant under a change of coordinates at infinity and that it agrees with the classical ADM mass in the smooth setting. We also provide an expression in terms of the Ricci tensor that agrees with the Ricci version of the ADM mass studied by Herzlich.
\end{abstract}

\input{sec_Introduction}
\input{sec_Preliminaries}
\input{sec_MassLowRegularity}
\input{sec_RicciMass}

\appendix
\bibliographystyle{alpha}
\bibliography{ref}
\end{document}

%% file: sec_Introduction.tex
\section{Introduction}
\label{sec:Introduction} 
The notion of mass is a central concept in mathematical general relativity. Depending on the setting the definition of mass might differ but one of the most well-known of these is the so-called ADM mass of an asymptotically Euclidean metric, given by Arnowitt, Deser, and Misner \cite{ADM59, ADM60, ADM61, ADM62}, see also \cite{ADM62}. 

An \textit{asymptotically Euclidean manifold} is, roughly speaking, a complete non-compact Riemannian manifold \( (M^n,g) \) such that outside of a compact set \( M \) is diffeomorphic to the complement of a closed ball in \( \mathbb R^n \) and such that \( g \) approaches the Euclidean metric \( \delta \) at infinity. The \textit{ADM mass} of such an \textit{asymptotically Euclidean metric} \( g \) is defined as
\begin{equation*}
    m_\text{ADM}(g)
    \coloneq \frac{1}{2(n-1)\omega_{n-1}} \lim_{R \to \infty} \int_{S_R} \bigl( \div_\delta(g) - d \tr_\delta(g) \bigr) (\nu_\delta) \,d\mu_\delta,
\end{equation*}
where \( \div_\delta \) respectively \( \tr_\delta \) is the divergence respectively the trace with respect to $\delta$, \( S_R \) is the \( (n-1) \)-sphere of radius \( R \) with outward pointing unit normal \( \nu_\delta \) and \( \omega_{n-1} \) is the \((n-1)\)-dimensional volume of the unit sphere \( S_1 \subseteq \mathbb R^n \). 
 
The \textit{positive mass conjecture} states that if the scalar curvature $\Scal^g$ of $(M,g)$ is non-negative, then \( m_\text{ADM}(g) \geq 0 \) with equality if and only if \( (M,g) \) is isometric to Euclidean space \( (\mathbb R^n,\delta) \). In 1979 Schoen and Yau \cite{SchoenYau, SchoenYau2} proved the positive mass conjecture for asymptotically Euclidean manifolds of dimension $n = 3$ using methods from the theory of minimal surfaces and in 1981 Witten \cite{Witten} proved the conjecture for all dimensions \( n \geq 3 \) under the assumption that the manifold \( M \) be spin. Recently the positive mass theorem has been proven using different methods, see for example Bray, Kazaras, Khuri and Stern \cite{BrayKazarasKhuriStern} and Agostiniani, Mazzieri and Oronzio \cite{AgostinianiMazzieriOronzio} for proofs involving so-called level set methods. To date, there are other generalizations of these results, for example the positive mass theorem has been proven to hold for manifolds $M$ having other asymptotic ends, in addition to a distinguished asymptotically Euclidean end, see Lesourd, Unger and Yau \cite{LesourdUngerYau}.

Recently, the field of mathematical general relativity has seen a surge of activity in the study of low regularity metrics. Progress in this area would allow the theory to describe (impulsive) gravitational waves, cf.\ \cite{LeFloch2011}, and other geometric singularities. Yet another reason to study low regularity metrics is related to the question of stability of the positive mass theorem, see for example Lee and Sormani \cite{LeeSormani2014}. Thus, there is a need for generalizations of classical concepts, including the notion of ADM mass, so that they might be applied in this low regularity setting. A notable result in this direction was obtained by Lee and LeFloch \cite{LeLFl} where they defined the notion of mass for metrics of local regularity $W^{1,n}_\text{loc} \cap C^0$ and proved a positive mass theorem in this setting. This in turn generalizes the earlier classical work of Bartnik \cite{Bartnik}, where the mass was defined for metrics of local regularity $W^{2,p}_\text{loc}$, for $p > n$.

The aim of this text is to generalize the notion of ADM mass to allow for asymptotically Euclidean metrics \( g \in W^{1,2}_\text{loc} \cap L^\infty \), a lower regularity than earlier definitions for asymptotically Euclidean metrics of Sobolev regularity. We do so by adapting the work of Gicquaud and Sakovich \cite{GicSak} carried out in the setting of weakly regular asymptotically hyperbolic manifolds. Using a suitable family of cutoff functions \( \{\chi_\alpha\}_{\alpha \geq 1} \), we define the \textit{weak ADM mass} to be the following limit:
\begin{equation*}
    m_\text{W}(g)
    \coloneq \frac{1}{2(n-1)\omega_{n-1}} \lim_{\alpha \to \infty} \int_{\mathbb R^n \setminus B_R} \bigl( \div_\delta(g) - d \tr_\delta(g) \bigr) (-\nabla^\delta\chi_\alpha) \,d\mu_\delta,
\end{equation*}
where \( B_R \) is the open ball of radius \( R \) and \( \nabla^\delta \chi_\alpha \) denotes the gradient of \( \chi_\alpha \) with respect to the Euclidean metric \( \delta \). In Theorem \ref{thm:WeakMassWellDefined} we prove that \( m_\text{W}(g) \) is finite and independent of the choice of \( \{\chi_\alpha\}_{\alpha \geq 1} \) and in Theorem \ref{thm:C2isSobolevMass} we show that it agrees with \( m_\text{ADM}(g) \) for \( C^2 \)-asymptotically Euclidean metrics. In Theorem \ref{thm:CoordinateInvariance} we show that the weak ADM mass does not depend on the choice of coordinates at infinity.  In addition, we show that our definition of mass can be expressed in terms the Ricci tensor as in Miao and Tam \cite{MiaoTam} and Herzlich \cite{Herz}. More specifically, in Theorem \ref{thm:MassRicci} and Theorem \ref{thm:RicciMassIsWeakMass} we prove the identity 
\begin{equation*}
    m_\text{W}(g)
    = \frac{-1}{(n-1)(n-2)\omega_{n-1}} \lim_{\alpha \to \infty} \int_{\mathbb R^n \setminus B_R} \biggl( \Ric^g{} - \frac{1}{2} \Scal^g{} g \biggr) (r \partial_r, -\nabla^\delta\chi_\alpha) \,d\mu_\delta,
\end{equation*}
where \( \Ric^g \) respectively \( \Scal^g \) is the Ricci tensor respectively the scalar curvature of \( g \) and \( r = \lvert x \rvert \) is the radial function on \( \mathbb R^n \). In Theorem \ref{thm:WeakMassAndWeakRicciMassEqual}, we show that in the \( C^2 \cap C^3_\text{loc} \)-setting this expression of the weak ADM mass agrees with the Ricci version of the ADM mass in Miao and Tam \cite{MiaoTam} and Herzlich \cite{Herz}. 

We would like to point out that we have not yet investigated the relation of our notion of mass to the notion of isoperimetric mass originally defined by Huisken in \cite{HuiskenIso}, see also \cite{JaureguiLeeUnger} for further references, as this would presumably require methods that are very different from the ones used in this article.

The outline of the paper is as follows. In Section \ref{sec:Preliminaries} we recall standard facts about weighted Sobolev spaces and asymptotically Euclidean metrics. In Section \ref{sec:Mass} we define and study the weak ADM mass of an asymptotically Euclidean metric of Sobolev regularity. In Section \ref{sec:RicciMass} we show that our weak mass can be expressed in terms of the Ricci tensor.

\textbf{Acknowledgements.} 
We would like to thank Anna Sakovich for helpful discussions during the writing of this paper. The first named author (SL) was supported by Grants-in-Aid for Scientific Research from the Ministry of Education, Culture, Sports, Science and Technology of Japan (No.\ JP21H05182). Parts of this work were completed while the second author (BM) was visiting the research group ``Geometric Analysis, Differential Geometry and Relativity Theory'' at the  University of Tübingen during Spring 2025. He thanks the group for their hospitality and Matariki Fellows for financial support during this visit. 

%% file: sec_Preliminaries.tex
\section{Preliminaries}
\label{sec:Preliminaries}

Throughout this paper we will consider Riemannian manifolds \( (M^n, g) \) of dimension \( n \geq 3 \). We denote the Lebesgue measure induced by \( g \) on \( M \) by \( d\mu_g \). The covariant derivative with respect to \( g \) is denoted by \( \nabla^g \) and all tensor norms with respect to \( g \) are denoted by \( \lvert \,\cdot\, \rvert_g \). The \( k \)th application of the covariant derivative \( \nabla^g \) to a tensor \( T \) is denoted by \( (\nabla^g)^{(k)}T \). We will also occasionally abuse notation slightly and write \( \nabla^g f \) in place of \( \grad^g(f) \). In Euclidean space, the sphere respectively the closed ball of radius \( R > 0 \) are denoted by \( S_R \) respectively \( B_R \), that is \( S_R \coloneq \{x \in \mathbb R^n : \lvert x \rvert = R\} \) respectively \( B_R \coloneq \{x \in \mathbb R^n: \lvert x \rvert \leq R\} \). The Euclidean metric is denoted by \( \delta \).

\subsection{Asymptotically Euclidean manifolds}

We now define reference manifolds and weighted Sobolev spaces on reference manifolds, as well as what it means for a metric to be asymptotically Euclidean. For this we mainly follow Lee and LeFloch \cite{LeLFl} and refer the reader there for further details. 

\begin{definition}
    \label{def:ReferenceManifold}
    Let \( M^n \) be a smooth manifold, \( K \subseteq M \) a compact set, \( R \geq 1 \) a radius and \( \Phi \colon M \setminus K \to \mathbb R^n \setminus B_R \) a diffeomorphism. The pair \( (M,\Phi) \) is then called a \textit{reference manifold} and \( \Phi \) a \textit{chart at infinity}. 
\end{definition}

A reference manifold can be equipped with background metric data, defined as follows. 

\begin{definition}
    \label{def:BackgroundData}
    Let \( (M^n,\Phi) \) be a reference manifold with \( \Phi \colon M \setminus K \to \mathbb R^n \setminus B_R \) its chart at infinity. A smooth metric \( h \) on \( M \) and a smooth function \( r \colon M \to (0,\infty) \) are said to be a \textit{background metric structure for \( (M,\Phi) \)} if \( \Phi \) is an isometry between \( (M\setminus K,h) \) and \( (\mathbb R^n \setminus B_R, \delta) \) and if for all \( x \in \mathbb R^n \setminus B_R \)
    \begin{equation*}
        (r \circ \Phi^{-1})(x) = \lvert x \rvert.
    \end{equation*}
\end{definition}

Because the set \( K \) in Definition \ref{def:ReferenceManifold} is compact, it follows that every reference manifold \( (M,\Phi) \) equipped with a background metric structure \( (h,r) \) is complete as a metric space.

\begin{remark}
    From now on, in all definitions and results we assume, unless otherwise stated, that \( (M^n,\Phi) \) is a reference manifold of dimension $n \geq 3$ with chart at infinity \( \Phi \colon M \setminus K \to \mathbb R^n \setminus B_R \) for some \( R \geq 1\) and is equipped with a background metric structure \( (h,r) \). We denote the covariant derivative with respect to the metric \( h \) by \( D \).
\end{remark}

For \( k \geq 0\) an integer and \( p \in [1, \infty] \), a tensor \( T \) defined on \( M \) is said to belong to the local \( L^p \) space \( L^p_\text{loc} \) if for any compact subset \( E \subset M \) its \( L^p \) norm
\begin{equation*}
    \lVert T \rVert_{L^p(E)}
    \coloneq
    \begin{cases}
        \displaystyle \biggl( \int_E \lvert T \rvert_h^p \,d\mu_h \biggr)^{1/p},     &p < \infty \\
        \displaystyle \operatorname*{\textnormal{ess sup}}_{E} {\lvert T \rvert_h},  &p = \infty
    \end{cases}
\end{equation*}
is finite. Similarly, \( T \) is said to belong to the local Sobolev space \( W^{k,p}_\text{loc} \) if \( D^{(l)}T \in L^p_\text{loc} \) for all \( 0 \leq l \leq k \). These spaces are independent of the choice of reference metric \( h \). We now recall the notion of weighted Sobolev spaces, cf.\ Bartnik \cite{Bartnik} and Lee and LeFloch \cite{LeFloch2011}.

\begin{definition}
    \label{def:WeightedSobolevSpaces}
    Let \( k \geq 0 \) be an integer, \( p \in [1, \infty] \) and \( \tau \in \mathbb R \). We define the \textit{weighted Sobolev space} \( W^{k,p}_{-\tau}(h,r) \) to be the set of tensors \( T \in W^{k,p}_\text{loc} \) whose weighted Sobolev norm
    \begin{equation*}
        \lVert T \rVert_{W^{k,p}_{-\tau}(h,r)} 
        \coloneq 
        \begin{cases}
            \displaystyle \sum_{l = 0}^k \biggl( \int_M r^{p(\tau+l)-n} \lvert D^{(l)}T \rvert_h^p \, d\mu_h \biggr)^{1/p},      &p < \infty \\
            \displaystyle \sum_{l = 0}^k \operatorname*{\textnormal{ess sup}}_M \bigl(r^{\tau+l} \lvert D^{(l)}T \rvert_h\bigr), &p = \infty
        \end{cases}
    \end{equation*}
    is finite. The \textit{weighted \( L^p \) space} \( L^p_{-\tau}(h,r) \) is then defined as \( L^p_{-\tau}(h,r) \coloneq W^{0,p}_{-\tau}(h,r) \).
\end{definition}

If no confusion can arise we write \( L^p_{-\tau} \) in place of \( L^p_{-\tau}(h,r) \) and \( W^{k,p}_{-\tau} \) in place of \( W^{k,p}_{-\tau}(h,r) \). 

\begin{remark}
    \label{rem:EuclideanSobolevs}
    At times we argue using tensors defined on \( \mathbb R^n \). For this reason we note that \( (\mathbb R^n,\text{id}) \) is a reference manifold that can be equipped with a background metric structure \( (\delta, \psi + (1-\psi) \lvert x \rvert) \), where \( \psi \colon \mathbb R^n \to [0,1) \) is any smooth function supported in \( B_1 \) such that \( \psi(0) > 0 \). We fix one such function \( \psi \) once and for all, and for \( k \geq 0 \) an integer, \( p \in [1,\infty] \) and \( \tau \in \mathbb R \), we define
    \begin{equation*}
        {}^\delta W^{k,p}_{-\tau} 
        \coloneq W^{k,p}_{-\tau} \bigl( \delta, \psi + (1-\psi) \lvert x \rvert \bigr). 
    \end{equation*}
    For a tensor \( T \in W^{k,p}_\textnormal{loc} \) which is supported on \( M \setminus K \), the tensor \( \Phi_* T\) is a priori only defined on \( \mathbb R^n \setminus B_R \). However, after extending \( \Phi_* T \) by \( 0 \) to all of \( B_R \) we have \( \Phi_* T \in {}^\delta W^{k,p}_\textnormal{loc} \) and
    \begin{equation*}
        \lVert T \rVert_{W^{k,p}_{-\tau}} = \lVert \Phi_* T \rVert_{{}^\delta W^{k,p}_{-\tau}}, 
    \end{equation*}
    so that in particular \( T \in W^{k,p}_{-\tau} \) if and only if \( \Phi_* T \in {}^\delta W^{k,p}_{-\tau} \).
\end{remark}

We now give the definition of asymptotically Euclidean metrics in low regularity.

\newpage

\begin{definition}
    \label{def:AESobolev}
    Let \( k \geq 0 \) be an integer, \( p \in [1, \infty] \) and \( \tau \in \mathbb R \). A complete Riemannian metric \( g \) on \( M \) is called \textit{\( W^{k,p}_{-\tau} \)-asymptotically Euclidean} if
    \begin{equation*}
        g, g^{-1} \in L^\infty_0, \quad
        e \coloneq g - h \in W^{k,p}_{-\tau},
    \end{equation*}
    and for some constant \( C > 1 \) we have
    \begin{equation*}
        C^{-1} h \leq g \leq C h, 
    \end{equation*}
    in the sense of quadratic forms.  Here \( g^{-1} \) is the inverse of \( g \), defined by \( g^{-1}(\alpha, \beta) \coloneq g(\alpha^{\sharp_g}, \beta^{\sharp_g}) \) for all \( 1 \)-forms \( \alpha \) and \( \beta \), with \( \sharp_g \colon T^*M \to TM \) the musical isomorphism induced by \( g \).
\end{definition}

Because of the equivalence at the end of Remark \ref{rem:EuclideanSobolevs}, a reference manifold \( (M^n,\Phi) \) equipped with a background metric structure \( (h,r) \) and an asymptotically Euclidean metric \( g \) induces a \emph{structure at infinity} as defined in \cite{Bartnik}. Lastly, we recall the standard definition of a \( C^k \)-asymptotically Euclidean metric.

\begin{definition}
    \label{def:SmoothAE}
    Let \( k \geq 0 \) be an integer and \( \tau > 0 \) a real number. A Riemannian metric \( g \in C^k_\text{loc}(M) \) is called \textit{\( C^k \)-asymptotically Euclidean of order \( \tau \)} if there is some constant \( C > 0 \) such that for \( e \coloneq g - h \) and all \( 0 \leq l \leq k \) we have
    \begin{equation*}
        \lvert D^{(l)} e \rvert_h \leq Cr^{-\tau-l}.
    \end{equation*}
\end{definition}

\subsection{Properties of weighted Sobolev spaces}

We now present weighted versions of certain standard results for Sobolev spaces. The reader is referred to \cite[Theorem \( 1.2 \)]{Bartnik} for further details and results. The following lemma, which is a straightforward consequence of Definition \ref{def:WeightedSobolevSpaces}, guarantees that functions of certain fall-off rates lie in Sobolev spaces with corresponding weights.

\begin{lemma}
    \label{lem:SmoothIsSobolev}
    If \( f \colon M \to \mathbb R \) is a continuous function that satisfies \( \lvert f \rvert \leq Cr^{-\tau} \) for some \( C > 0 \) and \( \tau \in \mathbb R \), then \( f \in L^p_{-w} \cap L^\infty_{-\tau} \) for all \( p \in [1,\infty) \) and \( w < \tau \).
\end{lemma}

The following weighted version of the Hölder inequality is a straightforward consequence of the classical Hölder inequality and Definition \ref{def:WeightedSobolevSpaces}, and can be proven by considering the cases $k = 0$ and $k > 0$ separately. See also \cite[Theorem $1.2$ ii)]{Bartnik}.

\begin{lemma}
    \label{lem:WeightedHölder}
    Suppose that \( k \geq 0 \) is an integer, that \( \tau_1, \tau_2 \in \mathbb R \) and that \(p_1, p_2, q \in [1,\infty] \) are such that \( p_1^{-1} + p_2^{-1} = q^{-1} \). If \( u_1 \in W^{k,p_1}_{-\tau_1} \) and \( u_2 \in W^{k,p_2}_{-\tau_2} \), then \( u_1 \otimes u_2 \in W^{k,q}_{-\tau_1 - \tau_2} \) and
    \begin{equation*}
        \lVert u_1 \otimes u_2 \rVert_{W^{k,q}_{-\tau_1-\tau_2}} 
        \leq C \lVert u_1 \rVert_{W^{k,p_1}_{-\tau_1}} \lVert u_2 \rVert_{W^{k,p_2}_{-\tau_2}},
    \end{equation*}
    for some constant \( C > 0 \) depending only on \( k \).
\end{lemma}

The next proposition is a special case of the Sobolev inequality.

\begin{lemma}
    \label{lem:WeightedSobolev}
    For an integer \( k \geq 0 \) and \( \tau, p \in \mathbb R \) such that \( p > n \) we have
    \begin{equation*}
        W^{k+1,p}_{-\tau} 
        \subseteq W^{k,\infty}_{-\tau}.
    \end{equation*}
\end{lemma}

\begin{proof}
    Denote the compact set of Definition \ref{def:ReferenceManifold} by \( K \), so that \( \Phi \colon M \setminus K \to \mathbb R^n \setminus B_R \). We first consider the case \( k = 0 \). Let \( u \in W^{1,p}_{-\tau} \) and let \( \phi \colon M \to [0,1] \) be a compactly supported smooth function such that \( \phi \equiv 1 \) in a neighbourhood of \( K \). The tensor \( u_1 \coloneq (1-\phi)u \) then lies in \( W^{1,p}_{-\tau} \) and has support in \( M \setminus K \). Extending \( \Phi_* u_1 \) by \( 0 \) to all of \( \mathbb R^n \), we have \( \Phi_* u_1 \in {}^\delta W^{1,p}_{-\tau} \) by Remark \ref{rem:EuclideanSobolevs}. An application of \cite[Theorem \( 1.2 \), \( (iv) \)]{Bartnik} shows that \( \Phi_*u_1 \in {}^\delta L^\infty_{-\tau} \) and again by Remark \ref{rem:EuclideanSobolevs} it follows that \( u_1 \in L^\infty_{-\tau} \). Next, we consider \( u_2 \coloneq \phi u \in W^{1,p}_\text{loc} \). Since \( u_2\in W^{1,p}_\text{loc} \), the Sobolev embedding theorem implies that \( u_2 \in L^\infty_\text{loc} \) and since \( u_2 \) is compactly supported it follows that \( u_2 \in L^\infty_{-\tau} \) as well. We thus conclude that \( u = u_1 + u_2 \in L^\infty_{-\tau} \). 
    
    For the general case we let \( u \in W^{k+1,p}_{-\tau} \), so that \( D^{(l)} u \in W^{k+1-l,p}_{-\tau-l} \subseteq W^{1,p}_{-\tau-l} \) for \( 0 \leq l \leq k \). But having just shown that \( W^{1,p}_{-\tau-l} \subseteq L^\infty_{-\tau-l} \), we conclude that $D^{(l)} u  \in L^\infty_{-\tau - l}$ for $l = 0,\dots, k$ and so $u \in W^{k,\infty}_{-\tau}$, which is what we wanted to show.
\end{proof}

Lastly, we need a weighted version of the Sobolev-Gagliardo-Nirenberg inequality. Our argument is based on the proof of \cite[Lemma $4.8$]{GicSak}, which is an analogous result in the asymptotically hyperbolic setting.

\begin{proposition}
    \label{prop:GagliardoNirenbergInterpolationFinal}
    Suppose that \( p \in [1,\infty) \), that \( q, s, t \in [1,\infty] \), that \( k > l \geq 0 \) are integers and that there is \( \theta \in [\frac{l}{k},1] \) such that
    \begin{equation*}
        \frac{1}{p} 
        = \frac{l}{n} + \theta \biggl(\frac{1}{q} - \frac{k}{n}\biggr) + \frac{1-\theta}{s}.
    \end{equation*}
    If \( k - l - \frac{n}{q} \) is a non-negative integer, then we additionally require that \( \theta < 1 \). Furthermore, suppose that \( \tau_0,\tau_1, \tau_2, \tau_3 \in \mathbb R \) satisfy
    \begin{equation*}
        \tau_0 < \min\bigl(\theta \tau_1 + (1-\theta) \tau_2,\tau_3\bigr). 
    \end{equation*}
    Then there is a constant \( C > 0 \), depending only on the parameters \((n,\tau_0,\tau_1,\tau_2,\tau_3,p,q,s,t,k,l,\theta) \) and the background metric structure \( (h,r) \), such that for any \( u \in W^{k,q}_\textnormal{loc} \) we have
    \begin{equation}
        \label{eq:GagliardoNirenbergInterpolation}
        \lVert D^{(l)} u \rVert_{L^p_{-\tau_0 - l}} 
        \leq C \bigl(\lVert u\rVert _{W^{k,q}_{-\tau_1}}^\theta\lVert u\rVert_{L^s_{-\tau_2}}^{1-\theta} + \lVert u\rVert_{L^t_{-\tau_3}}\bigr).
    \end{equation}
\end{proposition}

\begin{proof}
    Like in the proof of Lemma \ref{lem:WeightedSobolev} we let \( K \) be the compact set of Definition \ref{def:ReferenceManifold}. We let \( \phi \colon M \to [0,1] \) be a fixed compactly supported smooth function such that \( \phi \equiv 1 \) on \( K \). We can choose \( \phi \) in such a way that for all \( k \geq 0 \):
    \begin{equation*}
        \lVert \phi \rVert_{W^{k,\infty}_0} = \sum_{l = 0}^k \lVert D^{(l)}\phi \rVert_{L^\infty_{-l}} \leq C < \infty,
    \end{equation*}
    for some constant \( C > 0 \) depending only on the background metric structure \( (h,r) \).
    
    Defining \( u_1 \coloneq (1-\phi)u \) and \( u_2 \coloneq \phi u \), so that \( u = u_1 + u_2 \), we note that it suffices to show that the following two inequalities hold:
    \begin{align}
        \label{eq:SGN1}
        \lVert D^{(l)} u_1 \rVert_{L^p_{-\tau_0 - l}} & \leq C \bigl(\lVert u_1\rVert _{W^{k,q}_{-\tau_1}}^\theta\lVert u_1\rVert_{L^s_{-\tau_2}}^{1-\theta} + \lVert u_1\rVert_{L^t_{-\tau_3}}\bigr), \\
        \label{eq:SGN2}
        \lVert D^{(l)} u_2 \rVert_{L^p_{-\tau_0 - l}} & \leq C \bigl(\lVert u_2\rVert _{W^{k,q}_{-\tau_1}}^\theta\lVert u_2\rVert_{L^s_{-\tau_2}}^{1-\theta} + \lVert u_2\rVert_{L^t_{-\tau_3}}\bigr).
    \end{align}
    Indeed, if the above estimates hold, then for each non-negative integer \( m \leq k \) we have
    \begin{equation*}
        \begin{aligned}
            \lVert D^{(m)} u_2 \rVert_{L^q_{-\tau_1 - m}} 
            \leq C\sum_{i = 0}^m \lVert D^{(m-i)}\phi \otimes D^{(i)}u \rVert_{L^q_{-\tau_1 - m}}
            &\leq C\sum_{i = 0}^m \lVert D^{(m-i)}\phi\rVert_{L^\infty_{i-m}} \lVert D^{(i)}u   \rVert_{L^q_{-\tau_1-i}} 
            \\&\leq C\sum_{i = 0}^m \lVert D^{(i)}u \rVert_{L^q_{-\tau_1-i}}.
        \end{aligned}
    \end{equation*} 
    The rightmost sum above is nothing but \( C\lVert u \rVert_{W^{m,q}_{-\tau_1}} \). Summing over \( 0 \leq m \leq k \) we find that
    \begin{equation*}
        \lVert u_2 \rVert_{W^{k,q}_{-\tau_1}} \leq C\lVert u \rVert_{W^{k,q}_{-\tau_1}},
    \end{equation*}
    for some constant \( C \) depending only on the parameters \( (\tau_1,q,k) \) and the background metric structure \( (h,r) \). A similar argument shows that \( \lVert u_1 \rVert_{W^{k,q}_{-\tau_1}} \leq C \lVert u \rVert_{W^{k,q}_{-\tau_1}} \) and that
    \begin{equation*}
        \lVert u_i\rVert_{L^s_{-\tau_2}} \leq C\lVert u\rVert_{L^s_{-\tau_2}}, \quad
        \lVert u_i\rVert_{L^t_{-\tau_3}} \leq C\lVert u\rVert_{L^t_{-\tau_3}},
    \end{equation*}
    for \( i = 1,2 \) and some constant \( C > 0 \) depending only on \( (q,s,t,\tau_1,\tau_2,\tau_3) \) and the background metric structure \( (h,r) \). Thus
    \begin{equation*}
        \begin{aligned}
            \lVert D^{(l)} u \rVert_{L^p_{-\tau_0 - l}} 
            & \leq \lVert D^{(l)} u_1 \rVert_{L^p_{-\tau_0 - l}} + \lVert D^{(l)} u_2 \rVert_{L^p_{-\tau_0 - l}} \\
            & \leq C \bigl(\lVert u_1\rVert_{W^{k,q}_{-\tau_1}}^\theta\lVert u_1\rVert_{L^s_{-\tau_2}}^{1-\theta} + \lVert u_1\rVert_{L^t_{-\tau_3}} + \lVert u_2\rVert_{W^{k,q}_{-\tau_1}}^\theta\lVert u_2\rVert_{L^s_{-\tau_2}}^{1-\theta} + \lVert u_2\rVert_{L^t_{-\tau_3}}\bigr) \\
            & \leq 2C\bigl(\lVert u \rVert_{W^{k,q}_{-\tau_1}}^\theta\lVert u \rVert_{L^s_{-\tau_2}}^{1-\theta} + \lVert u\rVert_{L^t_{-\tau_3}}\bigr),
        \end{aligned}
    \end{equation*}
    which is nothing but \eqref{eq:GagliardoNirenbergInterpolation}.

    Since \( u_2 \) is compactly supported, the Sobolev-Gagliardo-Nirenberg inequality \cite[Equation $2.2$]{Nirenberg} together with a standard argument using a partition of unity subordinate to a collection of local coordinate charts that cover the support of \( \phi \) implies \eqref{eq:SGN2}, with the constant \( C \) depending only on the parameters \( (n,\tau_0,\tau_1,\tau_2,\tau_3,p,q,s,t,k,l,\theta) \) and the background metric structure \( (h,r) \). We now turn to proving \eqref{eq:SGN1}. 
    
    In the remainder of the proof we let \( A_\rho \coloneq \{x \in \mathbb R^n : \lvert x \rvert \in (\rho,2\rho]\} \) for each \( \rho > 0 \). We also define \( v_\rho(x) \coloneq v(\rho x) \) and use the weighted Lebesgue norms 
    \begin{equation*}
        \lVert v \rVert_{L^p_{-w}(E)}
        \coloneq
       \begin{cases}
            \displaystyle \biggl( \int_E \lvert x \rvert^{pw-n}\lvert v \rvert_\delta^p \,d\mu_\delta \biggr)^{1/p},       & p < \infty \\
            \displaystyle \operatorname*{\textnormal{ess sup}}_{E} \bigl( \lvert x \rvert^w \lvert v \rvert_\delta \bigr), & p = \infty
        \end{cases}
    \end{equation*}
    for tensors \( v \) on subsets \( E \subseteq \mathbb R^n \setminus B_1 \). Since \( u_ 1\) is supported in \( M \setminus K \), Remark \ref{rem:EuclideanSobolevs} implies that in order to show \eqref{eq:SGN1} it suffices to prove that for any tensor \( v \in {}^\delta W^{k,q}_\text{loc} \) whose support lies in \( \mathbb R^n \setminus B_R \) we have
    \begin{equation}
        \label{eq:SGNWeighted}
        \lVert D^{(l)} v \rVert_{L^p_{-\tau_0 - l}(\mathbb R^n \setminus B_R)} 
        \leq C \bigl(\lVert D^{(k)} v \rVert _{L^q_{-\tau_1 - k}(\mathbb R^n \setminus B_R)}^\theta\lVert v\rVert_{L^s_{-\tau_2}(\mathbb R^n \setminus B_R)}^{1-\theta} + \lVert v \rVert_{L^t_{-\tau_3}(\mathbb R^n \setminus B_R)}\bigr).
    \end{equation}
    For such a \( v \), the following scaling inequality holds, see \cite[equation \( 1.4 \)]{Bartnik}:
    \begin{equation}
        \label{eq:ScalingInequality}
        C^{-1}\rho^{w} \lVert v_\rho \rVert_{L^Q_{-w}(A_1)} 
        \leq \lVert v \rVert _{L^Q_{-w}(A_\rho)} 
        \leq C\rho^{w}\lVert v_\rho \rVert_{L^Q_{-w}(A_1)},
    \end{equation}
    the constant \( C \) depending only on the parameters \( w \) and \( Q \) but crucially not on \( v \) nor on \( \rho > 0 \). Defining \( R_i \coloneq 2^i R \) for \( i \geq 0 \), it follows that 
    \begin{equation*}
        \lVert D^{(l)} v\rVert_{L^p_{-\tau_0 - l}(\mathbb R^n \setminus B_R)}
        \leq \sum_{i = 0}^\infty \lVert D^{(l)}v\rVert_{L^p_{-\tau_0 - l}(A_{R_i})}
        \leq C\sum_{i = 0}^\infty R_i^{\tau_0 + l}\lVert (D^{(l)}v)_{R_i} \rVert_{L^p_{-\tau_0 - l}(A_1)},
    \end{equation*}
    where we have used the triangle inequality in the first inequality and \eqref{eq:ScalingInequality} with $w = \tau_0 + l$ and \( Q = p \) in the second. The chain rule implies that \( (D^{(l)}v)_R = R^{-l}D^{(l)}(v_R) \), so
    \begin{equation}
        \begin{aligned}
        \label{eq:IntermediateIneq}
            \lVert D^{(l)}v \rVert_{L^p_{-\tau_0 - l}(\mathbb R^n \setminus B_R)} 
            &\leq C\sum_{i = 0}^\infty R_i^{\tau_0 + l - l}\lVert D^{(l)}(v_{R_i})\rVert_{L^p_{-\tau_0 - l}(A_1)} 
            \\&= C\sum_{i = 0}^\infty R_i^{\tau_0}\lVert D^{(l)}(v_{R_i})\rVert_{L^p_{-\tau_0 - l}(A_1)}.
        \end{aligned}
    \end{equation}
    We now note that for any \( T \in L^Q_\textnormal{loc}(A_1) \) we have 
    \begin{equation*}
        C^{-1} \lVert T \rVert_{L^Q_{-w}(A_1)} \leq \lVert T \rVert_{L^Q_{-n/Q}(A_1)} \leq C \lVert T \rVert_{L^Q_{-w}(A_1)},
    \end{equation*} 
    for a constant \( C > 0 \) depending only on \( (Q,w,n) \). In the view of this, the unweighted Sobolev-Gagliardo-Nirenberg interpolation inequality \cite[Equation $2.2$]{Nirenberg} implies that for each \( T \in W^{k,q}_\textnormal{loc}(A_1) \) we have
    \begin{equation*}
        \begin{aligned}
            \lVert D^{(l)} T \rVert_{L^p_{-\tau_0 - l}(A_1)}
            \leq C\lVert D^{(l)} T \rVert_{L^p_{-p/n}(A_1)}
            & \leq C\bigl(\lVert D^{(k)} T \rVert_{L^q_{-q/n}(A_1)}^\theta \lVert T \rVert_{L^s_{-s/n}(A_1)}^{1-\theta} + \lVert T \rVert_{L^t_{-t/n}(A_1)}\bigr)
            \\& \leq C\bigl(\lVert D^{(k)} T \rVert_{L^q_{-\tau_1 - k}(A_1)}^\theta \lVert T \rVert_{L^s_{-\tau_2}(A_1)}^{1-\theta} + \lVert T \rVert_{L^t_{-\tau_3}(A_1)}\bigr),
        \end{aligned}
    \end{equation*}
    for some constant \( C > 0 \) that depends only on \( (n,\tau_0,\tau_1,\tau_2,\tau_3,p,q,s,t,k,l,\theta) \). Applying the above inequality to each term in the sum on the right-hand side of \eqref{eq:IntermediateIneq} then yields
    \begin{equation}
        \label{eq:AfterSGNStandard}
        \lVert D^{(l)}v \rVert_{L^p_{-\tau_0 - l}}
        \leq C \sum_{i = 0}^\infty \bigl(R_i^{\tau_0}\lVert D^{(k)}(v_{R_i})\rVert_{L^q_{-\tau_1 - k}(A_1)}^{\theta}\lVert v_{R_i} \rVert_{L^s_{-\tau_2}(A_1)}^{1-\theta} + R_i^{\tau_0}\lVert v_{R_i}\rVert_{L^t_{-\tau_3}(A_1)}\bigr),
    \end{equation}
    where the constant \( C \) above now depends on the parameter \( \tau_0 \) as well. We note that in \cite[Equation $2.2$]{Nirenberg}, the constant \( C \) depends on the domain as well. However, here the domain is always \( A_1 \). We use the chain rule again to conclude that \( D^{(k)}(u_\rho) = \rho^k(D^{(k)}u)_\rho \), so we find that
    \begin{equation}
        \label{eq:rNorm}
        \lVert D^{(k)}(v_{R_i})\rVert_{L^q_{-\tau_1 - k}(A_1)}
        = R_i^{k}\lVert (D^{(k)}v)_{R_i} \rVert_{L^q_{-\tau_1 - k}(A_1)} 
        \leq CR_i^{k - \tau_1 - k}\lVert D^{(k)}v\rVert_{L^q_{-\tau_1 - k}(A_{R_i})},
    \end{equation}
    where we have used \eqref{eq:ScalingInequality} in the last inequality, with \( w = \tau_1 + k \) and \( Q = q \). Two more applications of \eqref{eq:ScalingInequality}, once with \( w = \tau_2 \) and \( Q = s \) and once with \( w = \tau_3 \) and \( Q = t \), imply the bounds
    \begin{equation}
        \label{eq:qsNorm}
        \lVert v_{R_i}\rVert_{L^s(A_1)} 
        \leq CR_i^{-\tau_2}\lVert v\rVert_{L^s_{-\tau_2}(A_{R_i})}, \quad
        \lVert v_{R_i} \rVert_{L^t(A_1)} 
        \leq CR_i^{-\tau_3}\lVert v\rVert_{L^t_{-\tau_3}(A_{R_i})}.
    \end{equation}
    Combining \eqref{eq:AfterSGNStandard}, \eqref{eq:rNorm}, and \eqref{eq:qsNorm} we arrive at
    \begin{equation*}
        \begin{aligned}
            \lVert D^{(l)} v \rVert_{L^p_{-\tau_0 - l}}
            \leq C\biggl(\sum_{i = 0}^\infty R_i^{\tau_0 - \theta \tau_1 - \tau_2(1-\theta)}\lVert D^{(k)}v\rVert_{L^q_{-\tau_1 - k}(A_{R_i})}^{\theta}\lVert v \rVert_{L^s_{-\tau_2}(A_{R_i})}^{1-\theta} + R_i^{\tau_0 - \tau_3}\lVert v \rVert_{L^t_{-\tau_3}(A_{R_i})}\biggr).
        \end{aligned}
    \end{equation*}
    Recalling that \( \tau_0 < \min(\theta \tau_1 + (1-\theta) \tau_2, \tau_3) \) we find
    \begin{equation*}
        \lVert D^{(l)}v \rVert_{L^p_{-\tau_0 - l}}
        \leq C\sum_{i = 0}^\infty R_i^{-\epsilon}\lVert D^{(k)}v \rVert_{L^q_{-\tau_1 - k}(A_{R_i})}^{\theta}\lVert v \rVert_{L^s_{-\tau_2}(A_{R_i})}^{1-\theta} + C\sum_{i = 0}^\infty R_i^{-\epsilon}\lVert v \rVert_{L^t_{-\tau_3}(A_{R_i})},
    \end{equation*}
    for some \( \epsilon > 0 \). An application of the Hölder inequality yields
    \begin{equation}
        \label{eq:SGNalmost}
        \lVert D^{(l)}v \rVert_{L^p_{-\tau_0 - l}} 
        \leq
        \begin{aligned}[t]
             &C \biggl(\sum_{i = 0}^\infty R_i^{-\epsilon} \lVert D^{(k)}v \rVert_{L^q_{-\tau_1 - k}(A_{R_i})}\biggr)^{\theta} \biggl(\sum_{i = 0}^\infty R_i^{-\epsilon}\lVert v \rVert_{L^s_{-\tau_2}(A_{R_i})}\biggr)^{1-\theta} 
            \\&+ C\sum_{i = 0}^\infty R_i^{-\epsilon}\lVert v \rVert_{L^t_{-\tau_3}(A_{R_i})}.
        \end{aligned}
    \end{equation}
    If \( q = \infty \), the first factor of the first term in the right hand side of \eqref{eq:SGNalmost} can be bounded as
    \begin{equation*}
        \sum_{i = 0}^\infty R_i^{-\epsilon}\lVert D^{(k)}v \rVert_{L^q_{-\tau_1 - k}(A_{R_i})}
        \leq C\lVert D^{(k)}v \rVert_{L^\infty_{-\tau_1 - k}(\mathbb R^n \setminus B_R)}\sum_{i = 0}^\infty R_i^{-\epsilon}
        < C\lVert D^{(k)}v \rVert_{L^\infty_{-\tau_1 - k}(\mathbb R^n \setminus B_R)},
    \end{equation*}
    and if \( q < \infty \), then by the Hölder inequality we have
    \begin{equation*}
        \sum_{i = 0}^\infty R_i^{-\epsilon}\lVert D^{(k)}v \rVert_{L^q_{-\tau_1 - k}(A_{R_i})}
        \leq C \biggl( \sum_{i = 0}^\infty \lVert D^{(k)}v \rVert_{L^q_{-\tau_1 - k}(A_{R_i})}^q \biggr)^{1/q} 
        = C\lVert D^{(k)}v \rVert_{L^q_{-\tau_1 - k}(\mathbb R^n \setminus B_R)}.
    \end{equation*}
    A similar argument shows that for all \( 1 \leq s \leq \infty \) and all \( 1 \leq t \leq \infty \) we have
    \begin{equation*}
        \sum_{i = 0}^\infty R_i^{-\epsilon} \lVert v \rVert_{L^s_{-\tau_2}(A_{R_i})} \leq C\lVert v \rVert_{L^s_{-\tau_2}(\mathbb R^n \setminus B_R)}, \quad \sum_{i = 0}^\infty R_i^{-\epsilon} \lVert v \rVert_{L^t_{-\tau_3}(A_{R_i})} \leq C\lVert v \rVert_{L^t_{-\tau_3}(\mathbb R^n \setminus B_R)}.
    \end{equation*}
    In conclusion, \eqref{eq:SGNalmost} together with the above bounds implies 
    \begin{equation*}
        \lVert D^{(l)}v \rVert_{L^p_{-\tau_0 - l}(\mathbb R^n \setminus B_R)} 
        \leq C \bigl( \lVert D^{(k)}v \rVert_{L^q_{-\tau_1 - k}(\mathbb R^n \setminus B_R)}^{\theta}\lVert v \rVert_{L^s_{-\tau_2}(\mathbb R^n \setminus B_R)}^{1-\theta} + \lVert v \rVert_{L^t_{-\tau_3}(\mathbb R^n \setminus B_R)} \bigr),
    \end{equation*}
    which is nothing but \eqref{eq:SGNWeighted}. Thus \eqref{eq:SGN1} and hence also \eqref{eq:GagliardoNirenbergInterpolation} holds. 
\end{proof}

\subsection{Properties of asymptotically Euclidean metrics}
Using the results of the previous section we now derive some results on asymptotically Euclidean metrics that are needed in Section \ref{sec:Mass} and Section \ref{sec:RicciMass}. The following result tells us that if \( g \) is asymptotically Euclidean in the smooth sense, then \( g \) is indeed asymptotically Euclidean in the weak sense as well. 

\begin{proposition}
    \label{prop:SmoothIsSobolev}
    If \( g \) is a \( C^k \)-asymptotically Euclidean metric on \( M \) of order \( \tau > 0 \), then \( g \) is \( W^{k,p}_{-w} \)-asymptotically Euclidean for all \( p \in [1, \infty] \) and \( w < \tau \).
\end{proposition} 
\begin{proof}
    The asymptotic conditions of Definition \ref{def:SmoothAE} imply \( C^{-1} h \leq g \leq Ch \). By Lemma \ref{lem:SmoothIsSobolev} we have for \( e \coloneq g - h \) that \( D^{(l)}e \in L^p_{-w-l} \) for all integers \( 0 \leq l \leq k \). Hence \( e \in W^{k,p}_{-w} \).
\end{proof}

Given a \( W^{k,p}_{-\tau} \)-asymptotically Euclidean metric \( g \) on \( M \), in addition to the tensor field \( e = g - h \) in Definition \ref{def:AESobolev} we make frequent use of the tensor field  
\begin{equation}
    \label{eq:InvError}
    f  \coloneq g^{-1} - h^{-1},
\end{equation}
where \( h^{-1}(\alpha,\beta) \coloneq h(\alpha^{\sharp_h},\beta^{\sharp_h}) \) and \( \sharp_h \) is the musical isomorphism induced by \( h \). In local coordinates, the components of \( e \) and \( f \) are given by \( e_{ij} = g_{ij} - h_{ij} \) and \( f^{ij} = g^{ij} - h^{ij} \). The following lemma is  useful for bounding terms involving \( f \).

\begin{lemma}
    \label{lem:GeneralErrorEstimates}
    Suppose that \( p \in [1,\infty] \), that \( \tau \in \mathbb R \) and that \( g \) is a \( W^{1,p}_{-\tau} \)-asymptotically Euclidean metric on \( M \). Then in any local coordinate chart the components of \( e \) and \( f \) satisfy 
    \begin{equation}
        \label{eq:InvErrorOrder1}
        D_if^{jk} = - g^{jp}g^{kq}D_ie_{pq}.
    \end{equation}
    Moreover, there is a constant \( C > 1 \) such that
    \begin{equation}
        \label{eq:InvErrorNorm}
        \lvert f \rvert_h 
        \leq C\lvert e \rvert_h,
    \end{equation}
    and we also have
    \begin{equation}
        \label{eq:SobolevBounds}
        e,f \in W^{1,p}_{-\tau} \cap L^\infty_0.
    \end{equation}
\end{lemma}

\begin{proof}
    Equation \eqref{eq:InvErrorOrder1} and \eqref{eq:InvErrorNorm} and can be derived using exactly the same argument as in the derivation of \cite[equation $2.5$, equation $2.3$]{GicSak}. Next, due to the fact that \( h, h^{-1} \in L^\infty_0 \), the bound \( C^{-1} h \leq g \leq Ch \) shows that \( g, g^{-1} \in L^\infty_0 \). We thus find that \( e = g - h \in L^\infty_0 \) and \( f = g^{-1} - h^{-1} \in L^\infty_0 \). To show that \( f \in W^{1,p}_{-\tau} \) we note that the inequality \eqref{eq:InvErrorNorm} implies \( f \in L^p_{-\tau} \) and an application of Lemma \ref{lem:WeightedHölder} yields
    \begin{equation*}
        D_if^{jk} 
        = -\underbrace{g^{jp}g^{kq} \vphantom{e_{pq}}}_{L^\infty_0}\underbrace{D_ie_{pq}}_{L^p_{-\tau-1}} 
        \in L^p_{-\tau-1},
    \end{equation*}
    and so \( f \in W^{1,p}_{-\tau} \).
\end{proof}

The difference tensor \( \Gamma \) between \( \nabla^g \) and \( D \) is defined as \( \Gamma \coloneq \nabla^g - D \). There is a well-known expression for the components of \( \Gamma \) that mimics the expression for the connection coefficients of a metric connection. 

\begin{lemma}
    \label{lem:ChristofferSymbolsBounds}
    Suppose that \( p \in [1,\infty] \), that \( \tau \in \mathbb R \) and that \( g \) a \( W^{1,p}_{-\tau}\)-asymptotically Euclidean metric. In any coordinate chart the components of the difference tensor \( \Gamma \) are given by
    \begin{equation}
        \label{eq:DifferenceTensorComponents}
        \Gamma_{ij}^k = \frac{g^{kl}}{2}(D_i e_{jl} + D_j e_{il} - D_l e_{ij}).
    \end{equation}
    In particular, \( \Gamma \in L^p_{-\tau-1} \) and there is a constant \( C > 0 \) such that
    \begin{equation*}
        \lvert \Gamma \rvert_h 
        \leq C \lvert De \rvert_h.
    \end{equation*}
\end{lemma}

\begin{proof}
    Equation \eqref{eq:DifferenceTensorComponents} follows from a standard computation. Taking norms, using the Cauchy-Schwarz inequality and that \( \lvert g^{-1} \rvert_h \leq C \), we find that
    \begin{equation*}
        \lvert \Gamma \rvert_h 
        \leq C \lvert g^{-1} \rvert_h \lvert De \rvert_h
        \leq C \lvert De \rvert_h.
    \end{equation*}
\end{proof}

For the sake of simplicity we now introduce the following standard notation. For two tensors \( S \) and \( T \) on \( M \), we denote by \( S \star T \) any tensor constructed from the tensor product \( S \otimes T \) by raising or lowering indices and performing any number of contractions with respect to the metric \( h \), as well as any linear combination of such tensors.

The proofs of the following two propositions are merely computations. We have chosen to include them due to the absence of one term in \cite[Equation $2.3$]{LeLFl} only for the sake of completeness and we note that the missing term does not affect any of the arguments in \cite{LeLFl}.

\begin{proposition}
    \label{prop:FirstOrderRicci}
    Let \( g \) be a Riemannian metric on \( M \) such that \( g \in W^{2,2}_\textnormal{loc} \cap L^\infty_\textnormal{loc} \) and  \( g^{-1} \in L^\infty_\textnormal{loc} \). Then the Riemann curvature tensor \( \Riem^g \) and the Ricci curvature tensor \( \Ric^g \) of the metric \( g \) are well-defined and lie in \( L^1_\textnormal{loc} \). Additionally, in any local coordinate chart on $M$ we have
    \begin{equation*}
        {\Ric^g_{ij}} 
        =  \frac{g^{kl}}{2}(D_kD_ig_{lj} + D_kD_jg_{li} - D_kD_lg_{ij} - D_iD_jg_{kl}) + \mathcal Q^R_{ij},
    \end{equation*}
    where
    \begin{equation*}
            \mathcal Q^R_{ij} 
            = 
            \begin{aligned}[t]
                {\Ric^h_{ij}} + \Gamma_{ij}^u\Gamma_{vu}^v - \Gamma_{iv}^u\Gamma_{ju}^v 
                &+ \frac{D_ug^{ul}}{2}(D_i g_{lj} + D_jg_{li} - D_lg_{ij}) 
                \\&- \frac{D_ig^{ul}}{2}(D_u g_{lj} + D_jg_{lu} - D_lg_{uj}).
        \end{aligned}
    \end{equation*}
\end{proposition}

\begin{proof}
    In a geodesic normal coordinate chart about a central point \( p \in M \), the Christoffel symbols \( {}^h \Gamma_{ij}^k \) of the metric \( h \) vanish at \( p \). Using the standard formula for the Riemannian curvature tensor, a calculation shows that
    \begin{equation}
        \label{eq:RiemannComparison}
        {\Riem^g_{ijk}}{}^l - {\Riem^h_{ijk}}{}^l
        = D_i\Gamma_{jk}^l - D_j\Gamma_{ik}^l + \Gamma_{jk}^u\Gamma_{iu}^l - \Gamma_{ik}^u \Gamma_{ju}^l.  
    \end{equation}
    Since the right and left hand sides are expressions which only involve tensors, we conclude that the above equality holds in any coordinate chart, so \( \Riem^g = \Riem^h + D\Gamma + \Gamma \star \Gamma \). But \( \Riem^h \in L^1_\text{loc} \) and Lemma \ref{lem:ChristofferSymbolsBounds} gives the inclusions
    \begin{equation*}
        \underbrace{\Gamma}_{L^2_\text{loc}} \star \underbrace{\Gamma}_{L^2_\text{loc}} \in L^1_\text{loc}, \quad
        D\Gamma = \underbrace{De \vphantom{g}}_{L^2_\text{loc}} \star \underbrace{De \vphantom{g}}_{L^2_\text{loc}} + \underbrace{g}_{L^2_\text{loc}} \star \underbrace{D^{(2)}e \vphantom{g}}_{L^2_\text{loc}} \in L^1_\text{loc},
    \end{equation*}
    from which we conclude that \( \Riem^g \in L^1_\text{loc} \). Since \( g^{-1} \in L^\infty_\text{loc} \) we find that \( \Ric^g \in L^1_\text{loc} \) as well. 
    
    Contracting the first and last index of \eqref{eq:RiemannComparison}, we obtain
    \begin{equation*}
        {\Ric}^g_{ij} = {\Ric}^h_{ij} + D_v\Gamma_{ij}^v - D_i\Gamma_{vj}^v + \Gamma_{ij}^u\Gamma_{vu}^v - \Gamma_{iv}^u \Gamma_{ju}^v.
    \end{equation*}
    Next, a calculation and a use of \eqref{eq:InvErrorOrder1} gives
    \begin{equation*}
        D_v\Gamma_{ij}^v - D_i\Gamma_{vj}^v 
        = 
        \begin{aligned}[t]
            &\frac{g^{vl}}{2}(D_vD_jg_{li} + D_iD_lg_{vj} - D_vD_lg_{ij} - D_iD_jg_{lv}) \\
            &+ \frac{D_vg^{vl}}{2}(D_i g_{lj} + D_jg_{li} - D_lg_{ij}) - \frac{D_ig^{vl}}{2}(D_v g_{lj} + D_jg_{lv} - D_lg_{vj}).
        \end{aligned}
    \end{equation*}
    Thus upon defining \( \mathcal Q^R_{ij}  \) as in Proposition \ref{prop:FirstOrderRicci},
    we find that
    \begin{equation*}
        \Ric^g_{ij} 
        = \frac{g^{kl}}{2}(D_kD_ig_{lj} + D_kD_jg_{li} - D_kD_lg_{ij} - D_iD_jg_{kl}) + \mathcal Q^R_{ij},
    \end{equation*}
    with \( \mathcal Q^R = g^{-1} \star g^{-1} \star De \star De \), which we wanted to show.
\end{proof}

\begin{proposition}
    \label{prop:FirstOrderScalar} 
    Let \( g \) be a Riemannian metric on \( M \) such that \( g \in W^{2,2}_\textnormal{loc} \cap L^\infty_\textnormal{loc} \) and  \( g^{-1} \in L^\infty_\textnormal{loc} \). Then the scalar curvature \( \Scal^g \) is well-defined and lies in \( L^1_\textnormal{loc} \). Moreover, 
    \begin{equation*}
        \Scal^g
        = \div_h(V) + \mathcal Q^S,
    \end{equation*}
    where the function \( \mathcal Q^S \) and the components of the vector field \( V \) in any local coordinate chart are given by
    \begin{align*}
        V^i & = g^{ij}g^{kl}(D_ke_{lj} - D_je_{kl}) = (g^{ij}g^{kl} - g^{ik}g^{jl})D_ke_{jl}, \\
        \mathcal Q^S & = \Scal^h + f^{ij}\Ric^h_{ij} + g^{ij}(\Gamma_{ij}^u\Gamma_{vu}^v - \Gamma_{vi}^u\Gamma_{ju}^v) - (D_vg^{ij})\Gamma_{ij}^v + (D_ig^{ij})\Gamma_{vj}^v.
    \end{align*}
\end{proposition}

\begin{proof}
    \( \Scal^g \) is well-defined since \( g \in W^{2,2}_\text{loc} \). By Proposition \ref{prop:FirstOrderRicci} we know that \( \Ric^g \in L^1_\text{loc} \), which combined with \( g^{-1} \in L^\infty_\text{loc} \) implies that \( \Scal^g \in L^1_\text{loc} \). 
    
    Like in the proof of Proposition \ref{prop:FirstOrderRicci} we can deduce the following equality
    \begin{equation*}
        {\Ric}^g_{ij} = {\Ric}^h_{ij} + D_v\Gamma_{ij}^v - D_i\Gamma_{vj}^v + \Gamma_{ij}^u\Gamma_{vu}^v - \Gamma_{vi}^u \Gamma_{ju}^v,
    \end{equation*}
    which holds in any local coordinate chart. Multiplying both sides of the above identity by $g^{ij}$ and summing the resulting identity over the indices $i,j$, we find
    \begin{equation}
        \label{eq:ScalarComparison1}
        \Scal^g = \Scal^h + f^{ij}\Ric^h_{ij} + g^{ij}(\Gamma_{ij}^u\Gamma_{vu}^v - \Gamma_{vi}^u\Gamma_{ju}^v) + g^{ij}(D_v \Gamma_{ij}^v - D_i\Gamma_{vj}^v).
    \end{equation}
    Next, we note that
    \begin{equation}
        \begin{aligned}
        \label{eq:PartialDiff}
            g^{ij}(D_v \Gamma_{ij}^v - D_i\Gamma_{vj}^v)
            &= D_v (g^{ij}\Gamma_{ij}^v) - D_i(g^{ij}\Gamma_{vj}^v) - (D_vg^{ij})\Gamma_{ij}^v + (D_ig^{ij})\Gamma_{vj}^v
            \\& =D_i (g^{vj}\Gamma_{vj}^i - g^{ij}\Gamma_{vj}^v) - (D_vg^{ij})\Gamma_{ij}^v + (D_ig^{ij})\Gamma_{vj}^v.
        \end{aligned}
    \end{equation}
    But using Lemma \ref{lem:ChristofferSymbolsBounds} and a change of index, we observe that the first term in the right hand side of \eqref{eq:PartialDiff} is the divergence of \( V \):
    \begin{equation}
        \begin{aligned}
        \label{eq:FunkyTerm}
            g^{vj}\Gamma_{vj}^i - g^{ij}\Gamma_{vj}^v
            & = g^{vj} \frac{g^{iu}}{2} (D_vg_{ju} + D_jg_{vu} - D_ug_{vj}) - g^{ij} \frac{g^{uv}}{2} (D_v g_{ju} - D_jg_{vu} - D_ug_{vj})
            \\& = g^{ij}g^{uv}(D_ug_{vj} - D_jg_{vu})
            \\&= V^i.
        \end{aligned}
    \end{equation}
    Combining \eqref{eq:ScalarComparison1}, \eqref{eq:PartialDiff} and \eqref{eq:FunkyTerm}, we arrive at
    \begin{equation*}
        \Scal^g 
        = \div_h V + \bigl(\Scal^h + f^{ij}\Ric^h_{ij} + g^{ij}(\Gamma_{ij}^u\Gamma_{vu}^v - \Gamma_{vi}^u\Gamma_{ju}^v) - (D_vg^{ij})\Gamma_{ij}^v + (D_ig^{ij})\Gamma_{vj}^v\bigr).
    \end{equation*}
    Comparing terms, we see that \( \Scal^g - \div_h (V) = \mathcal{Q}^S \), which is what we wanted to show.
\end{proof}

The following lemma relates the volume forms of two metrics on $M$.

\begin{lemma}
    \label{lem:VolumeComparison} 
    Let \( g_1, g_2 \) be two Riemannian metrics on \( M \) satisfying \( C_0^{-1} h < g_i < C_0h \) for \( i = 1, 2 \) in the sense of bilinear forms, where \( C_0 > 0 \) is some constant. Then
    \begin{equation*}
        \lvert \sqrt{\det(g_1)} - \sqrt{\det (g_2)} \rvert 
        < C\lvert g_1 - g_2 \rvert_h,
    \end{equation*}
    for some constant \( C > 0 \) depending only on \( C_0,h \) and the dimension \( n \). 
\end{lemma}

\begin{proof}
    The function $\det \colon \mathrm{GL}(\mathbb R^n) \to \mathbb R$ is smooth and for any constant $C_0 > 0$ the set $S_{C_0} \coloneq \{L \in \mathrm{GL}(\mathbb R^n) : C_0^{-1}I_{\mathbb R^n} \leq L \leq C_0I_{\mathbb R^n}\}$ is compact. Hence there is a constant $C > 0$ depending only on the dimension $n$ and the constant $C_0$ such that for all $L_1, L_2 \in S_{C_0}$ we have 
    \begin{equation*}
        \lvert\sqrt{\det(L_1)} - \sqrt{\det(L_2)}\rvert \leq C\lvert L_1 - L_2\rvert_\delta,
    \end{equation*}
    where $\lvert L \rvert_\delta \coloneq \bigl\lvert \sum_{i,j = 1}^n L_{ij}^2 \bigr\rvert^{1/2}$. Working in an orthonormal frame with respect to $g_1$, the above implies that there is a constant $C > 0$, depending only on $n$ and $C_0$ such that whenever $C_0^{-1}g_1 < g_2 < C_0g_1$, we have
    \begin{equation*}
        \lvert\sqrt{\det(g_2)} - \sqrt{\det(g_1)}\rvert \leq C\lvert g_2 - g_1\rvert_{g_1}.
    \end{equation*}
    But since $C_0^{-1}h < g_1 < C_0h$ it follows that $\lvert g_2 - g_1\rvert_{g_1} < C\lvert g_2 - g_1\rvert_h$ as well.
\end{proof}

The following technical result is needed in Section \ref{sec:RicciMass}, c.f.\ \cite[Lemma \( 4.8 \)]{GicSak}.

\begin{proposition}
    \label{prop:CubicIntegrability} 
    For all \( \tau \in \mathbb R \) there is a constant \( C \), depending only on \( n,\tau \) and the background metric structure $(h,r)$, such that for all \( u \in W^{2,2}_\textnormal{loc} \) we have
    \begin{align*}
        \lVert Du \rVert_{L^3_{-1-\tau}} & \leq C \lVert Du \rVert_{W^{1,2}_{-1-\tau}} && \text{when \( 3 \leq n \leq 6 \), and} \\
        \lVert Du \rVert_{L^3_{-1-\frac{2\tau}{3}}} & \leq C \bigl(\lVert u \rVert_{W^{2,2}_{-\tau}} + \lVert u \rVert_{L^\infty_0} \bigr) && \text{when $n \geq 7$.} 
    \end{align*}
\end{proposition}

\begin{proof}
    First, let \( 3 \leq n \leq 6 \). Arguing as in the proof of Lemma \ref{lem:WeightedSobolev} and Proposition \ref{prop:GagliardoNirenbergInterpolationFinal}, we may apply the Sobolev inequality \cite[Theorem $1.2$, $iv)$]{Bartnik} to \( Du \) to conclude that
    \begin{equation*}
        \lVert Du \rVert_{L^{np/(n-p)}_{-\tau-1}} 
        \leq C \lVert Du \rVert_{W^{1,q}_{-\tau-1}} 
        < \infty,
    \end{equation*}
    for each \( p \) and \( q \) that satisfy \( 1 \leq p < n \) and \( p \leq q \leq \frac{np}{n-p} \). But it is easy to see that since \( n \leq 6 \), these inequalities hold for \( q = 2 \) and \( \frac{np}{n-p} = 3 > 2 = q \), since in that case
    \begin{equation*}
        p = \frac{3n}{n+3} \leq \frac{3 \cdot 6}{3 + 6} = 2 = q < 3 = \frac{np}{n-p}.
    \end{equation*}
    Next, let \( n > 6 \). In this case, we use Proposition \ref{prop:GagliardoNirenbergInterpolationFinal} with \( k = 2 \), \( l = 1 \), \( p = 3 \), \( q = t = 2 \), \( s = \infty \), \(\tau_1 = \tau_3 = \tau \) and \( \tau_2 = 0 \). With these choices, the constant \( \theta \) is required to satisfy \( 1/2 < \theta < 1 \). But here
    \begin{equation} \label{eq:ThetaForm}
        \theta 
        = \frac{2}{3}\frac{n-3}{n-4},
    \end{equation}
    so we have \( 1/2 < \theta < 1\) precisely when \( n > 6 \). Proposition \ref{prop:GagliardoNirenbergInterpolationFinal} therefore implies that
    \begin{equation} \label{eq:LemmaInterpol}
        \lVert Du \rVert_{L^3_{-v-1}} 
        \leq C\bigl(\lVert u \rVert_{W^{2,2}_{-\tau}}^\theta \lVert u \rVert_{L^\infty_0}^{1-\theta} + \lVert u\rVert_{L^2_{-\tau}}\bigr),
    \end{equation}
    for every weight \( v \) which satisfies $v < \tau$ and $v < \theta \tau + (1-\theta)\cdot 0 = \theta \tau$. As \( \theta < 1 \) it suffices that the second condition holds. This is satisfied by \( v = \frac{2\tau}{3} \), since \eqref{eq:ThetaForm} and \( n > 6 \) imply
    \begin{equation*}
        v
        = \frac{2}{3} \tau
        < \frac{2}{3}\frac{n-3}{n-4}\tau
        = \theta \tau.
    \end{equation*}
    With this choice of \( v \), an application of the bound \( a^\theta b^{1-\theta} \leq \theta a + (1-\theta) b \) in \eqref{eq:LemmaInterpol} gives us
    \begin{equation*}
        \begin{aligned}
            \lVert Du \rVert_{L^3_{-1-\frac{2\tau}{3}}}
            &\leq C\bigl( \lVert u \rVert_{W^{2,2}_{-\tau}}^\theta \lVert u\rVert_{L^\infty_{0}}^{1-\theta} + \lVert u\rVert_{L^2_{-\tau}}\bigr)
            \\&\leq
            C\bigl(\theta \lVert u \rVert_{W^{2,2}_{-\tau}} + (1-\theta) \lVert u\rVert_{L^\infty_0} + \lVert u\rVert_{L^2_{-\tau}}\bigr)
            \\&\leq
            C\bigl( \lVert u \rVert_{W^{2,2}_{-\tau}} + \lVert u \rVert_{L^\infty_0} \bigr).
     \end{aligned}
     \end{equation*}
\end{proof}

%% file: sec_MassLowRegularity.tex
\section{The mass of asymptotically Euclidean manifolds in low regularity}
\label{sec:Mass}
In this section, we show that the notion of ADM mass can be generalized to the case of asymptotically Euclidean metrics in \( W^{1,2}_\text{loc} \cap L^\infty_0 \) with suitable fall-off. For this, we use the method introduced by Gicquaud and Sakovich \cite{GicSak} in the asymptotically hyperbolic setting. The rough idea is to replace the definition of the mass as a boundary integral, which is potentially ill-defined in this low regularity setting, by a bulk integral using cutoff functions. This allows us to lower the assumed regularity even further compared to the works of Bartnik \cite{Bartnik}, and Lee and LeFloch \cite{LeLFl} that are also concerned with asymptotically Euclidean metrics of weak regularity.

We recall the classical definition of the ADM mass.

\begin{definition}
    \label{def:MassSmooth}
    Let \( g \) be a \( C^1 \)-asymptotically Euclidean metric on \( M \) of order \( \tau > \frac{n-2}{2} \) as in Definition \ref{def:SmoothAE} such that \( g \in C^2_\text{loc} \) and \( \Scal^g \in L^1(M) = L^1_{-n} \). The \textit{ADM-mass} of \( (M,g) \) is given by
    \begin{equation*}
        m_\text{ADM}(g) 
        \coloneq \frac{1}{2(n-1)\omega_{n-1}} \lim_{R \to \infty} \int_{\mathbb{S}_R} \bigl(\div_h(g) - d\tr_h(g)\bigr)(\nu_h) \,d\mu_h,
    \end{equation*}
    where \( \omega_{n-1} \) is the measure of the unit sphere in \( \mathbb R^n \) and \( \nu_h \) is the outward pointing unit normal to \( \mathbb S_R = \{p \in M: r(p) = R\} \) with respect to \( h \).
\end{definition}

In the case when the local regularity of the metric $g$ is merely \( W^{1,2}_\text{loc} \cap L^\infty_0 \) there are two points in this definition that need to be addressed. First, the integrand \( (\div_h(e)- d\tr_h(e))(\nu_h) \) might not be well-defined on \( \mathbb S_R = \{p \in M: r(p) = R\} \), since the \( (n-1) \)-spheres have \( n \)-dimensional Hausdorff measure zero. Similarly to \cite{GicSak}, we remedy this with the help of cutoff functions, more specifically by having integrals over spheres being replaced by integrals over annuli and unit normals of spheres being replaced by the gradients of cutoff functions. To deal with the fact that the scalar curvature may not be well-defined in this regularity class, we define scalar curvature as a distribution, c.f.\ Lee and LeFloch \cite[Definition $2.1$]{LeLFl}.

\begin{definition}
    \label{def:WeakScalarCurvature}
    Let \( g \) be a Riemannian metric on \( M \) such that \( g \in W^{1,2}_\text{loc} \cap L^\infty_\text{loc} \) and \( g^{-1} \in L^\infty_\text{loc} \). We define the \textit{scalar curvature distribution of \( g \)} as 
    \begin{equation}
        \label{eq:DistributionDefinition}
        \ScalOp{\phi} 
        \coloneq \int_M \bigl( V(-\phi) + \phi \mathcal Q^S  \bigr) \, d\mu_h 
        = \int_M \bigl( (g^{ij}g^{kl} - g^{ik}g^{jl})D_ke_{jl}(-D_i \phi) + \phi \mathcal Q^S \bigr) \, d\mu_h,
    \end{equation}
    for any locally Lipschitz function \( \phi \colon M \to \mathbb R \) with compact support, where the function \( \mathcal Q^S \) and the vector field \( V \) are as in Proposition \ref{prop:FirstOrderScalar}.
\end{definition}
When $g$ is as in Definition \ref{def:WeakScalarCurvature} we have \( V \in L^1_\text{loc} \) and \( \mathcal Q^S \in L^1_\text{loc} \). These inclusions combined with \( \phi \in W^{1,\infty}_\text{loc} \) having compact support imply that the integrand in the right hand side of \eqref{eq:DistributionDefinition} is in \( L^1_{-n} \), hence \( \ScalOp{\phi} \) is well-defined and finite. Of course, when additionally \( g \in W^{2,2}_\text{loc} \) and \( \phi \in W^{1,\infty}_\text{loc} \) is compactly supported, then
\begin{equation*}
    \ScalOp{\phi} = \int_M \phi \Scal^g \, d\mu_h.
\end{equation*}
Inspired by the approach of Gicquaud and Sakovich \cite[Section \( 4.1 \)]{GicSak}, we use the following definition of the distributional scalar curvature being integrable.

\begin{definition}
    \label{def:WeakIntegrableScalarCurvature}
    Let \( g \) be a Riemannian metric on \( M \) such that \( g \in W^{1,2}_\text{loc} \cap L^\infty_\text{loc} \) and \( g^{-1} \in L^\infty_\text{loc} \). We say that \( g \) has \textit{distributional scalar curvature in \( L^1 \)} if there is a function \( \Sc \in L^1_{-n} \) such that for all compactly supported and locally Lipschitz functions \( \phi \)
    \begin{equation*}
        \ScalOp{\phi} = \int_M \phi \Sc \, d\mu_h.
    \end{equation*}
\end{definition}

\begin{remark} 
    There are other ways of generalizing the notion of scalar curvature to a low regularity setting and we by no means claim that Definition \ref{def:WeakScalarCurvature} and Definition \ref{def:WeakIntegrableScalarCurvature} are the most general or appropriate ones. We refer the reader to Gicquaud and Sakovich \cite[Remark $4.7$]{GicSak} for a more thorough discussion regarding different possible generalizations and what properties any such generalization should posses. Although their discussion is in the asymptotically hyperbolic setting, analogous statements are true in the asymptotically Euclidean case. 
\end{remark}

If \( g \in W^{2,2}_\text{loc} \) and \( (M,g) \) is \( W^{1,2}_{-\tau} \)-asymptotically Euclidean, then \( \Scal^g \in L^1_{-n} \) if and only if $(M,g)$ has distributional scalar curvature in \( L^1 \) in the sense of Definition \ref{def:WeakIntegrableScalarCurvature}. In this case the function \( \Sc \) in Definition \ref{def:WeakIntegrableScalarCurvature} is nothing but \( \Sc = \Scal^g \). The following example shows that a metric \( g \) with distributional scalar curvature in \( L^1 \) does not necessarily satisfy \( g \in W^{2,2}_\text{loc} \).

\begin{example}
    \label{ex:ScalExample}
    There is a metric \( g \) on \( \mathbb R^3 \) which is \( {^\delta W^{1,2}_{-\tau}} \)-asymptotically Euclidean for all \( \tau \in \mathbb R \), \( g \notin W^{2,2}_\textnormal{loc}(\mathbb R^3) \), and such that \( g \) has distributional scalar curvature in \( L^1 \).
\end{example}

\begin{proof}
    Let \( \alpha \in (0,1/2) \) and define \( \varphi: B_2 \to \mathbb R \) by \( \varphi(\vec x) := |\vec x|^\alpha\). Since \( \alpha > 0 \) we have
    \begin{equation*}
        \varphi \in L^\infty(B_2), \quad \varphi \in W^{1,2}(B_2), \quad \Delta^\delta \varphi \in L^1(B_2).
    \end{equation*}
    Since \( \alpha < 1/2 \), it follows that \( \varphi \notin W^{2,2}(B_2) \). We now let \( \psi \colon \mathbb R^3 \to [0,1]\) be a compactly supported smooth function such that \( \psi \equiv 1 \) on \( B_1 \) while \( \psi \equiv 0 \) on \( \mathbb R^3 \setminus B_{3/2} \). Then, the function \( \varphi_1 \coloneq \varphi \psi \) is compactly supported and satisfies
    \begin{equation*}
        \varphi_1 \in {^\delta L^\infty_0}, \quad \varphi_1 \in {^\delta W^{1,2}_{-\tau}}, \quad \varphi_1 \notin W^{2,2}_\text{loc}(\mathbb R^3), \quad \Delta^\delta \varphi_1 \in L^1_{-3},
    \end{equation*}
    for all \( \tau \in \mathbb R \). We define the metric \( g \coloneq (1 + \varphi_1) \delta \). Since \( \varphi_1 \geq 0 \) and \( \varphi_1 \in L^\infty_0 \) we have \( \delta \leq g \leq C\delta \) for some constant \( C > 0 \). Due to the equality \( g - \delta = \varphi_1 \delta  \) it follows that \( g \in {^\delta W^{1,2}_{-\tau}} \) for all \( \tau > 0 \) while \( g \notin W^{2,2}_\text{loc}(\mathbb R^3) \).

    Next we note that on the set \( \mathbb R^3 \setminus \{\vec 0\} \), \( g \) is smooth and so \( \Scal^g \) is defined, smooth and satisfies
    \begin{equation*}
        \Scal^g 
        = \frac{3}{2}\underbrace{(1 + \varphi_1)^{-3}}_{\in  L^\infty_0} \underbrace{\lvert D \varphi_1 \rvert^2}_{ \in L^1_{-3} } - \underbrace{\vphantom{\lvert D \varphi_1 \rvert^2}  2(1 + \varphi_1)^{-2}}_{\in L^\infty_0} \underbrace{\vphantom{\lvert D \varphi_1 \rvert^2} \Delta^\delta \varphi_1}_{\in L^1_{-3}} \in L^1_{-3}.
    \end{equation*}
    To see that \( g \) has integrable scalar curvature in the sense of Definition \ref{def:WeakIntegrableScalarCurvature}, we let \( v \in W^{1,\infty}_\text{loc} \) be compactly supported, \( V \in L^1_\text{loc}(\mathbb R^3) \) and \( \mathcal Q^S \in L^1_\text{loc}(\mathbb R^3) \) be as in Proposition \ref{prop:FirstOrderScalar} and we claim that we may define \( \mathcal S^g = \Scal^g \) almost everywhere. To see this, we note that
    \begin{equation}
        \label{eq:IntScal1}
        \begin{split}
            \int_{\mathbb R^3} v \Scal^g \, d\mu_\delta 
            & = \lim_{\epsilon \to 0} \int_{\mathbb R^3 \setminus B_\epsilon} v \Scal^g \, d\mu_\delta \\
            & = \lim_{\epsilon \to 0} \biggl(\int_{\mathbb R^3 \setminus B_\epsilon} v \div_\delta (V) \, d\mu_\delta + \int_{\mathbb R^3 \setminus B_\epsilon} v \mathcal Q^S \, d\mu_\delta\biggr) \\
            & = \lim_{\epsilon \to 0} \biggl(-\int_{\partial B_\epsilon} v \frac{V^ix_i}{\lvert \vec x \rvert} \, d\mu_\delta - \int_{\mathbb R^3 \setminus B_\epsilon} V \cdot v \, d\mu_\delta + \int_{\mathbb R^3 \setminus B_\epsilon} v \mathcal Q^S \, d\mu_\delta\biggr) \\
            & = - \int_{\mathbb R^3} V \cdot v \,  d\mu_\delta + \int_{\mathbb R^3} v \mathcal Q^S \, d\mu_\delta - \lim_{\epsilon \to 0} \int_{\partial B_\epsilon} v \frac{V^ix_i}{\lvert \vec x \rvert} \, d\mu_\delta,
        \end{split}
    \end{equation}
    where in the first line above we have used that \( \Scal^g \in L^1_\text{loc}(\mathbb R^n) \) and the dominated convergence theorem, in the second line we have applied Proposition \ref{prop:FirstOrderScalar}, in the third line we have applied the divergence theorem and lastly in the fourth line we have used that \( V, \mathcal Q^S \in L^1_\text{loc}(\mathbb R^3) \) and applied the dominated convergence theorem once again. On \( \partial B_\epsilon \) it holds \( \lvert V \rvert_\delta  \leq C\epsilon^{\alpha - 1} \) for some constant \( C > 0 \) independent of \( \epsilon \). Thus
    \begin{equation*}
        \lim_{\epsilon \to 0} \, \biggl\lvert \int_{\partial B_\epsilon} v \frac{V^ix_i}{\lvert \vec x \rvert} \, d\mu_\delta\biggr\rvert
        \leq \limsup_{\epsilon \to 0} C\epsilon^{3-1}\epsilon^{\alpha - 1} 
        \leq \limsup_{\epsilon \to 0} C\epsilon^{1 + \alpha} 
        = 0,
    \end{equation*}
    where in the last step above we have again used that \( \alpha > 0 \). The above limit combined with \eqref{eq:IntScal1} shows that \( g \) has distributional scalar curvature in \( L^1 \).
\end{proof}

We are now ready to define the weak mass of a \( W^{1,2}_{-\tau} \)-asymptotically Euclidean metric. Note that the following definition is analogous to \cite[Equation $4.5$]{GicSak}.

\begin{definition}
    \label{def:WeakMass} 
    Let \( g \) be a \( W^{1,2}_{-\tau} \)-asymptotically Euclidean metric on \( M \) with weight \( \tau \in \mathbb R \) and let \( \{ \chi_\alpha \colon M \to \mathbb R \}_{\alpha \geq 1} \) be a sequence of compactly supported locally Lipschitz functions such that
    \begin{equation*}
        \sup_\alpha \lVert \chi_\alpha \rVert_{W^{1,\infty}_0}  
        = \sup_\alpha \sup_M \bigl( \lvert \chi_\alpha \rvert + r \lvert D\chi_\alpha \rvert \bigr) 
        \leq C,
    \end{equation*}
    for some constant \( C > 0 \), and such that the sequence \( \{\chi_\alpha^{-1}(1)\}_{\alpha \geq 1} \) of subsets of \( M \) on which \( \chi_\alpha \equiv 1 \) satisfies \( K \subseteq \chi_\alpha^{-1}(1) \subseteq \chi_{\alpha+1}^{-1}(1) \) for all \( \alpha \geq 1 \) and \( \bigcup_{\alpha \geq 1} \chi_\alpha^{-1}(1) = M \). We then define the \textit{weak ADM mass of \( g \) with respect to \( \Phi \) and \( \{\chi_\alpha\}_{\alpha \geq 1} \)} by
    \begin{equation*}
        \begin{aligned}
            m_\text{W} \bigl( g,\Phi,\{\chi_\alpha\}_{\alpha \geq 1} \bigr) 
            \coloneq{} 
            &\frac{1}{2\omega_{n-1}(n-1)}\lim_{\alpha \to \infty}\int_M \bigl(\div_h(e) - d\tr_h(e)\bigr)(-D\chi_\alpha) \,d\mu_h \\
            ={} &\frac{1}{2\omega_{n-1}(n-1)}\lim_{\alpha \to \infty}\int_M (h^{ij}h^{kl} - h^{ik}h^{jl})D_ke_{jl}(-D_i\chi_\alpha) \,d\mu_h,
        \end{aligned}
    \end{equation*}
    whenever this limit exists. Here we have abused notation slightly by denoting the gradient of $\chi_\alpha$ with respect to the metric \( h \) by \( D\chi_\alpha \).
\end{definition}

\begin{remark}
    \label{rem:TypicalCutoffs}
    A typical sequence that satisfies the conditions of Definition \ref{def:WeakMass} is given by
    \begin{equation*}
        \chi_\alpha(p) = 
        \begin{cases}
            1,                     & r(p) \leq \alpha \\
            2-\frac{r(p)}{\alpha}, & \alpha < r(p) < 2\alpha \\
            0,                     & r(p) \geq 2\alpha
        \end{cases},
    \end{equation*}
    which means that the integration in Definition \ref{def:WeakMass} is carried out on a relatively ``thick'' set, while in the classical definition of mass, integration is carried out on a very ``thin'' set. This might seem unnatural, but we show that these definitions are equivalent in the higher regularity setting and the condition \( \sup_{\alpha} \lVert \chi_\alpha \rVert_{W^{1,\infty}_0} < \infty \) seems to be necessary in order to prove the optimal results. 
\end{remark}

\subsection{Well-definedness of the weak ADM mass}

We now show the weak ADM mass is well-defined and equal to the classical ADM mass when the metric is \( C^2_\text{loc} \). We begin by showing that the integral in Definition \ref{def:WeakMass} converges and does not depend on the choice of cutoff functions. The following lemma is useful for showing that many of the integrals we encounter in the arguments below vanish. 

\begin{lemma}
    \label{lem:IntegralsToZero}
    Suppose that \( w \in \mathbb R \), that \( Y \in L^1_{w-n} \) and that \( \{X_\alpha\}_{\alpha \geq 1} \) is a sequence of tensor fields of the same rank as $Y$ such that
    \begin{equation*}
        \sup_\alpha \lVert X_\alpha \rVert_{L^\infty_{-w}} \leq C, \quad 
        \lim_{\alpha \to \infty} X_\alpha(p) = 0 \text{ for all \( p \in M \)},
    \end{equation*}
    for some constant \( C > 0 \). Then
    \begin{equation*}
        \lim_{\alpha \to \infty} \int_M \langle Y, X_\alpha \rangle_h \,d\mu_h = 0.
    \end{equation*}
\end{lemma}
\begin{proof}
    By assumption, we have point-wise convergence \( \lim_{m \to \infty} \langle Y, X_\alpha \rangle_h \rvert_p = 0 \) for every \( p \in M \). Moreover, the bound \(  \sup_\alpha \lVert X_\alpha \rVert_{L^\infty_{-w}} \leq C \) together with the inclusion \( Y \in L^1_{w-n} \) imply \( \lvert \langle Y, X_\alpha \rangle_h \rvert \leq \lvert Y \rvert_h \lvert X_\alpha \rvert_h \leq Cr^{-w} \lvert Y \rvert_h \in L^1_{-n} \). In other words, the functions \( \langle Y, X_\alpha \rangle \) are dominated by the integrable function \( Cr^{-w} \lvert Y \rvert_h \) and converge pointwise to \( 0 \). An application of the dominated convergence theorem then yields the desired limit.
\end{proof}

The next theorem is analogous to the first part of \cite[Proposition $4.1$]{GicSak}. 

\begin{theorem}
    \label{thm:WeakMassWellDefined}
    Let \( g \) be a \( W^{1,2}_{-(n-2)/2} \)-asymptotically Euclidean metric that has distributional scalar curvature in \( L^1 \) and let \( \{\chi_\alpha\}_{\alpha \geq 1}\) be a sequence of cutoff functions as in Definition \ref{def:WeakMass}. Then the weak ADM mass of \( g \) is finite and independent of the choice of cutoff functions.
\end{theorem}
\begin{proof}
    We prove existence and independence simultaneously. Let $K$ be the compact set in Definition \ref{def:ReferenceManifold} and let \( \Sc \) be as in Definition \ref{def:WeakIntegrableScalarCurvature}. Fixing a locally Lipschitz function \( \phi \colon M \to \mathbb R \) that vanishes on \( K \) and is equal to \( 1 \) outside some compact set that contains \( K \), we define another sequence of functions \( \{\overline \chi_\alpha\}_{\alpha \geq 1} \) by \( \overline \chi_\alpha \coloneq \phi\chi_\alpha \) for each \( \alpha \in \mathbb N \). By construction, each member \( \overline \chi_\alpha \) of this sequence is locally Lipschitz with compact support. Letting the scalar curvature act on a member \( \overline{\chi}_\alpha \) of this sequence as in Definition \ref{def:WeakScalarCurvature} we have
    \begin{equation*}
        \ScOp{\overline{\chi_\alpha}}
        = \int_M (g^{ij}g^{kl} - g^{ik}g^{jl})D_ke_{jl}(-D_i\overline \chi_\alpha) \,d\mu_h + \int_M \overline \chi_\alpha \mathcal Q^S \,d\mu_h.
    \end{equation*}
    The supports of \( D\chi_\alpha \) and \( D\phi \) are disjoint for all indices \( \alpha \) larger than some \( \alpha_0 \). Thus for all \( \alpha \geq \alpha_0 \) we have \( D \overline\chi_\alpha = D\chi_\alpha + D\phi \), after which we can rewrite the above equality as follows:
    \begin{equation}
        \begin{aligned}
        \label{eq:ScalActingOnCutRearranged}
            \int_M (h^{ij}h^{kl} & - h^{ik}h^{jl})D_ke_{jl}(-D_i \chi_\alpha) \,d\mu_h 
            \\&= 
            \begin{aligned}[t]
                \ScOp{\overline{\chi_\alpha}}
                &- \int_M \overline \chi_\alpha \mathcal Q^S \,d\mu_h
                - \int_M (g^{ij}g^{kl} - g^{ik}g^{jl})D_ke_{jl}(-D_i \phi) \,d\mu_h
                \\& - \int_M (g^{ij}g^{kl} - h^{ij}h^{kl} + h^{ik}h^{jl} - g^{ik}g^{jl})D_ke_{jl}(-D_i \chi_\alpha) \,d\mu_h.
            \end{aligned}   
        \end{aligned}
    \end{equation}
    In the limit \( \alpha \to \infty \) the left-hand side above becomes a constant multiple of the weak ADM mass. Therefore we study what happens with the right-hand side when \( \alpha \to \infty \). The next to last integral on the right-hand side is independent of \( \alpha \) and finite since \( D\phi \) has compact support. Moreover, since \( g \) is \( W^{1,2}_{1-n/2} \)-asymptotically Euclidean we have \( \mathcal Q^S \in L^1_{-n} \). Combining this with the pointwise convergence \( (\overline\chi_\alpha - \phi) \to 0 \) and the bound
    \begin{equation*}
        \sup_\alpha \lVert \overline \chi_\alpha - \phi\rVert_{L^\infty_0} 
        \leq \lVert \phi \rVert_{L^\infty_0} + \sup_\alpha \lVert \overline \chi_\alpha \rVert_{L^\infty_0} 
        < \infty,
    \end{equation*}
    we conclude after an application of Lemma \ref{lem:IntegralsToZero} that
    \begin{equation}
        \label{eq:MassWellDefinedLimit1}
        \lim_{\alpha \to \infty} \int_M \overline \chi_\alpha \mathcal Q^S \,d\mu_h 
        = \int_M \phi \mathcal Q^S \,d\mu_h
        < \infty.
    \end{equation}
    We now consider the leftmost integral in \eqref{eq:ScalActingOnCutRearranged}. Due to \( \phi \in L^\infty_0 \) and  \( \overline \chi_\alpha \in L^\infty_0 \) and the pointwise convergence \( \overline{\chi_\alpha} = \chi_\alpha \phi \to \phi \), an application of the dominated convergence theorem shows that
    \begin{equation}
        \label{eq:MassWellDefinedLimit2}
        \lim_{\alpha \to \infty} \ScOp{\overline \chi_\alpha}
        = \ScOp{\phi}
        < \infty.
    \end{equation} 
    Next we note that
    \begin{equation*}
        g^{ij}g^{kl} - h^{ij}h^{kl} 
        = \underbrace{(g^{ij} - h^{ij})}_{L^2_{1-n/2}}\underbrace{g^{kl}\vphantom{)}}_{L^\infty_0} 
        + \underbrace{h^{ij}\vphantom{)}}_{L^\infty_0}\underbrace{(g^{kl} - h^{kl})}_{L^2_{1-n/2}} 
        \in L^2_{1-n/2}.
    \end{equation*}
    By a similar argument one finds \( g^{ik}g^{jl} - h^{ik}h^{jl} \in L^2_{1-n/2} \) and so
    \begin{equation*}
        \underbrace{(g^{ij}g^{kl} - h^{ij}h^{kl} + h^{ik}h^{jl} - g^{ik}g^{jl})\vphantom{e_j}}_{L^2_{1-n/2}}
        \underbrace{D_ke_{jl}}_{L^2_{-n/2}} 
        \in L^1_{1-n}.
    \end{equation*}
    Due to the point-wise convergence \( D\chi_\alpha \to 0 \) and the bound \( \sup_\alpha \lVert D\chi_\alpha \rVert_{L^\infty_{-1}} < \infty \) we conclude after another application of Lemma \ref{lem:IntegralsToZero} that
    \begin{equation}
        \label{eq:MassWellDefinedLimit3}
        \lim_{\alpha \to \infty}\int_M (g^{ij}g^{kl} - h^{ij}h^{kl} + h^{ik}h^{jl} - g^{ik}g^{jl})D_ke_{jl}(-D_i \chi_\alpha) \,d\mu_h 
        = 0.
    \end{equation}
    Letting \( \alpha \to \infty \) in \eqref{eq:ScalActingOnCutRearranged} and using \eqref{eq:MassWellDefinedLimit1}, \eqref{eq:MassWellDefinedLimit2} and \eqref{eq:MassWellDefinedLimit3} we conclude that
    \begin{equation}
        \label{eq:NotZeroMass}
        \begin{aligned}
            2(n-1)\omega_{n-1} &m_\text{W} \bigl(g,\Phi, \{\chi_\alpha\}_{\alpha \geq 1} \bigr)
            \\ & = \ScOp{\phi} - \int_M \phi \mathcal Q^S \,d\mu_h - \int_M (g^{ij}g^{kl} - g^{ik}g^{jl})D_ke_{jl}(-D_i \phi) \,d\mu_h.
        \end{aligned}
    \end{equation}
    This equality implies that the weak ADM mass of \( g \) with respect to the sequence \( \{\chi_\alpha\}_{\alpha \geq 1} \) is well-defined. Moreover the sequence \( \{\chi_\alpha\}_{\alpha \geq 1} \) does not appear on the right-hand side, hence \( m_\text{W}(g,\Phi,\{\chi_\alpha\}_{\alpha \geq 1}) \) does not depend on the choice of cutoff functions. 
\end{proof}

\begin{remark}
    \label{rem:NotZeroMass}
    Note that since the function \( \phi \) in the proof of Theorem \ref{thm:WeakMassWellDefined} is not compactly supported, the right hand side of \eqref{eq:NotZeroMass} does not necessarily vanish.
\end{remark}

Since the weak mass is independent of the choice of cutoff-functions, we simply denote it by \( m_\text{W}(g,\Phi) \). When \( (M,g) \) is \( C^1 \)-asymptotically Euclidean, as in Definition \ref{def:SmoothAE}, Proposition \ref{prop:SmoothIsSobolev} implies that \( g \) is \( W^{1,2}_{-(n-2)/2} \)-asymptotically Euclidean. Moreover if \( g \in C^2_\text{loc} \) and \( \Scal^g \in L^1_{-n} \), then due to the discussion below Definition \ref{def:WeakIntegrableScalarCurvature}, \( g \) has distributional scalar curvature in \( L^1 \) with \( \Sc = \Scal^g \), where \( \Sc \) is as in Definition \ref{def:WeakIntegrableScalarCurvature}. By Theorem \ref{thm:WeakMassWellDefined}, the weak ADM mass of such a metric \( g \) is well-defined and independent of the choice of cutoff-functions. 

We now show that in this case, \( m_\text{W}(g,\Phi) = m_\text{ADM}(g,\Phi) \), see also the second part of \cite[Proposition $4.1$]{GicSak} for an analogous result in the asymptotically hyperbolic setting. 

\begin{theorem}
    \label{thm:C2isSobolevMass}
    Let \( g \) be a \( C^1 \)-asymptotically Euclidean metric of order \( \tau > \tfrac{n-2}{2} \) with \( g \in C^2_\textnormal{loc} \) and \( \Scal^g \in L^1_{-n} \). Then weak ADM mass equals the classical ADM mass:
    \begin{equation*}
        m_\textnormal{W} (g,\Phi)
        = m_\textnormal{ADM}(g,\Phi). 
    \end{equation*}
\end{theorem}

\begin{proof}
    For integers \( \alpha \geq 1 \), we define the cutoff functions \( \chi_\alpha \) as in Remark \ref{rem:TypicalCutoffs}. We note in particular that \( D\chi_\alpha = -1_{\{\alpha < r < 2\alpha\}}\alpha^{-1} Dr \). By Definition \ref{def:WeakMass} we find that
    \begin{align*}
        m_\text{W}(g,\Phi) 
        & = \frac{1}{2\omega_{n-1}(n-1)}\lim_{\alpha \to \infty}\int_M \bigl( \div_h(e) - d\tr_h(e) \bigr)(-D\chi_\alpha) \,d\mu_h \\
        & = \frac{1}{2\omega_{n-1}(n-1)}\lim_{\alpha \to \infty}\frac{1}{\alpha}\int_{\alpha \leq r \leq 2\alpha} \bigl( \div_h(e) - d\tr_h(e) \bigr)(Dr) \,d\mu_h.
    \end{align*}
    Since the outward pointing unit normal to \( \mathbb S_R = \{p \in M: r(p) = R\} \) with respect to \( h \) is \( \nu_h = \frac{Dr}{\lvert Dr \rvert_h} = Dr \) and since the classical ADM mass of \( (M,g) \) is well-defined, the above can be written as
    \begin{align*}
        m_\text{W}(g, \Phi) 
        & = \lim_{\alpha \to \infty}\frac{1}{\alpha}\int_\alpha^{2\alpha}\biggl(\frac{1}{2\omega_{n-1}(n-1)} \int_{\mathbb S_s} \bigl( \div_h(e) - d\tr_h(e) \bigr)(Dr) \,d\mu_h \biggr) \,ds \\
        & = \lim_{\alpha \to \infty} \frac{1}{\alpha} \int_\alpha^{2\alpha} \bigl(m_\text{ADM}(g,\Phi) + o(1)\bigr) \,ds
        \\& = m_\text{ADM}(g,\Phi).
    \end{align*}
\end{proof}

\subsection{Coordinate invariance of the weak mass}
In this section we provide conditions under which the weak ADM mass does not depend on the choice of chart at infinity. Our starting point is the following reformulation of a result due to Bartnik \cite{Bartnik} in terms of weighted Sobolev spaces on \( \mathbb R^n \), recall Remark \ref{rem:EuclideanSobolevs} for the definition of \( {}^\delta W^{k,p}_{-\tau} \). We also refer the reader to \cite[Lemma $1$]{Chrusciel} for a related result.

\begin{proposition}
    \label{prop:UniquenessOfInfinity}
    Suppose that \( n < p < \infty \), that \( \tau > 0 \), that \( (M,\Phi) \) and \( (M,\widetilde \Phi) \) are background manifolds with background metric structures \( (h,r) \) and \( (\tilde h, \tilde r) \), respectively and that \( g \) is a Riemannian metric on \( M \) that is both \( W^{1,p}_{-\tau}(h,r) \) and \( W^{1,p}_{-\tau}(\tilde h,\tilde r) \)-asymptotically Euclidean. Then there is an isometry \( G \colon \mathbb R^n \to \mathbb R^n \) such that for \( i = 1,\dots, n \) we have
    \begin{equation*}
        \nabla^\delta (F^i - x^i) \in {}^\delta W^{1,p}_{-\tau},
        \quad \nabla^\delta \bigl( (F^{-1})^i - x^i \bigr) \in {}^\delta W^{1,p}_{-\tau},
    \end{equation*}
    where \( F \coloneq G \circ \widetilde \Phi \circ \Phi^{-1} \) and \( x^i \) is the \( i \)th Euclidean coordinate function.
\end{proposition}

\begin{remark}
    \label{rem:AbuseOfNotation}
    We are abusing notation slightly since the maps \( F,F^{-1} \) are only defined outside of compact subsets of \( \mathbb R^n \) and so really the above should be phrased using the maps \( (1-\psi)F \) and \( (1-\psi)F^{-1} \) instead, where \( \psi \) is as in Remark \ref{rem:EuclideanSobolevs}. For the sake of brevity, we ignore this technicality.
\end{remark}

\begin{proof}
    Recall that  \( (M,\Phi) \), \( (h,r) \) and \( (M,\widetilde \Phi) \), \( (\tilde h, \tilde r) \) induce structures at infinity as defined in \cite[Definition $2.1$]{Bartnik} since both \( \Phi_*g - \delta \) and \( \widetilde \Phi_*g - \delta \) lie in \( {}^\delta W^{1,p}_{-\tau} \). We can thus apply \cite[Corollary $3.2$]{Bartnik} with $(M,\widetilde \Phi)$ and $(M,\Phi)$ in place of $(M,\Phi)$ and $(M,\Psi)$ to obtain an isometry \( G \colon \mathbb R^n \to \mathbb R^n \) such that, defining \( F \coloneq G \circ \widetilde\Phi \circ \Phi^{-1} \), we have
    \begin{equation*}
        F^i - x^i = 
        (G \circ \widetilde\Phi \circ \Phi^{-1} - \text{id})^i
        \in {}^\delta W^{2,p}_{1-\tau},
    \end{equation*}
    and hence also \( \nabla^\delta (F^i - x^i) \in {}^\delta W^{1,p}_{-\tau} \) for \( i = 1,\dots, n\). We now prove that \( \nabla^\delta ((F^{-1})^i - x^i) \in {}^\delta W^{1,p}_{-\tau} \) for \( i = 1,\dots, n\) as well.
    
    Applying \cite[Corollary $3.2$]{Bartnik} again, but interchanging the roles of $\Phi$ and $\widetilde \Phi$, we conclude that there is an isometry \( \widetilde G \colon \mathbb R^n \to \mathbb R^n \) such that, defining \( \widetilde F \coloneq \widetilde G \circ \Phi \circ \widetilde \Phi^{-1} \), we have
    \begin{equation}
        \label{eq:IntroduceSecondG}
        \widetilde F^i - x^i
        = (\widetilde G \circ \Phi \circ \widetilde \Phi^{-1} - \text{id})^i 
        \in {}^\delta W^{2,p}_{1-\tau}.
    \end{equation}
    Letting $d$ denote the total derivative and \( ( \,\cdot\, )_j^i \) the entries of the matrix corresponding to an element in \( \mathrm{GL}(\mathbb R^n) \), we now prove that \( d\widetilde G = dG^{-1} \). We start by noting that Lemma \ref{lem:WeightedSobolev} implies that \( F^i - x^i \) and \( \widetilde F^i - x^i \) are in \( {}^\delta W^{1,\infty}_{1-\tau} \), so
    \begin{align*}
        \bigl(d(F - \text{id})\bigr)^i_j = (dG \circ d\widetilde\Phi \circ d\Phi^{-1} - \text{id})^i_j &\in {}^\delta L^\infty_{-\tau} \\
        \bigl(d(\widetilde F - \text{id})\bigr)^i_j = (d\widetilde G \circ d\Phi \circ d\widetilde \Phi^{-1} - \text{id})^i_j &\in {}^\delta L^\infty_{-\tau},
    \end{align*}
    for all $1 \leq i,j \leq n$. Since \( \tau > 0 \), these inclusions imply
    \begin{equation}
        \label{eq:FirstLimit}
        \lim_{p \to \infty} (d F)^i_j 
        = \lim_{p \to \infty} (d \widetilde F)^i_j 
        = \delta^i_j.
    \end{equation}
    Now, since inversion is continuous as a map \( \mathrm{GL}(\mathbb R^n) \to \mathrm{GL}(\mathbb R^n) \),
    it follows that
    \begin{equation*}
        \lim_{p \to \infty} (d\widetilde \Phi \circ d\Phi^{-1} \circ d\widetilde G^{-1})^i_j 
        =\lim_{p \to \infty} (d \widetilde F^{-1} )^i_j 
        = (\delta^{-1})_i^j
        = \delta^i_j.
    \end{equation*}
    Moreover, since function composition \( \circ: \mathrm{GL}(\mathbb R^n) \times \mathrm{GL}(\mathbb R^n) \to \mathrm{GL}(\mathbb R^n) \) is continuous, the above implies
    \begin{equation}
        \label{eq:SecondLimit}
        \lim_{p \to \infty} (d\widetilde \Phi \circ d\Phi^{-1})^i_j 
        = \lim_{p \to \infty} (d \widetilde F^{-1} \circ d\widetilde G)^i_j 
        = (\mathrm{id} \circ d\widetilde G )^i_j
        = d\widetilde G^i_j.
    \end{equation}
    We can therefore calculate as follows:
    \begin{equation*}
        (dG \circ d\widetilde G)^i_j 
        = (dG)^i_k (d\widetilde G)^k_j 
        = \lim_{p \to \infty} (dG)^i_k (d\widetilde \Phi \circ d\Phi^{-1})^k_j 
        = \lim_{p \to \infty} (d F)^i_j
         = \delta^i_j,
    \end{equation*}
    where we have used \eqref{eq:SecondLimit} in the second equality and \eqref{eq:FirstLimit} in the last. It follows that \( d\widetilde G = dG^{-1} \) as claimed. But then \eqref{eq:IntroduceSecondG} implies
    \begin{equation*}
        (dG^{-1} \circ d\Phi \circ d\widetilde \Phi^{-1} - \text{id})^i_j 
        =(d\widetilde G \circ d\Phi \circ d\widetilde \Phi^{-1} - \text{id})^i_j
        =\bigl(d(\widetilde F - \mathrm{id})\bigr)^i_j
        \in {}^\delta W^{1,p}_{-\tau}
    \end{equation*}
    Summing up, we have
    \begin{align*}
        (dF^{-1} - \text{id})^i_j 
        & = ( d\Phi \circ d\widetilde \Phi^{-1} \circ dG^{-1} -  \text{id})^i_j \\
        & = \bigl(dG\circ (dG^{-1} \circ d\Phi \circ d\widetilde \Phi^{-1} -  \text{id}) \circ dG^{-1}\bigr)^i_j \\
        & = \underbrace{(dG)^i_k \vphantom{)^k_j}}_{{}^\delta W^{1,\infty}_0} \underbrace{(dG^{-1} \circ d\Phi \circ d\widetilde \Phi^{-1} -  \text{id})^k_l \vphantom{)^k_j}}_{{}^\delta W^{1,p}_{-\tau}} \underbrace{(dG^{-1})^l_j}_{{}^\delta W^{1,\infty}_0} \in {}^\delta W^{1,p}_{-\tau},
    \end{align*}
    which we wanted to show.
\end{proof}

In the view of the above result, in order to understand how the mass might differ between different background metric structures, we need to understand how the mass is affected if the chart at infinity \( \Phi \colon M\setminus K \to \mathbb R^n \setminus B_R \) is composed with an isometry \( G \colon \mathbb R^n \to \mathbb R^n \) or with an ``almost identity'' \( F \colon \mathbb R^n \setminus B_R \to \mathbb R^n \). 

\begin{remark}
    \label{rem:EverythingIsReference}
    We note that whenever \( (M,\Phi) \) is a reference manifold and \( \Psi \colon \mathbb R^n \setminus B_R \to \mathbb R^n\) is a diffeomorphism such that $\mathbb R^n \setminus \text{image}(\Psi)$ is compact, then $(M, \Psi \circ \Phi)$ is a reference manifold. This is because in this case the composition \( \Psi \circ \Phi \) is a diffeomorphism as well and we can choose a bigger compact set \( K' \supset K \) and a radius $R' > R$ such that \( (\Psi \circ \Phi)(M \setminus K') = \mathbb R^n \setminus B_{R'} \).
\end{remark}
We need the following technical result, which shows that comparable background metric structures give rise to the same weighted Sobolev spaces. 

\begin{proposition}
    \label{prop:ComparingBMS}
    Let $p \in [1,\infty]$ and $\tau \in \mathbb R$. Suppose that \( (M,\Phi) \) and \( (M, \widetilde \Phi) \) are background manifolds and that \( (h,r) \) and \( (\tilde h, \tilde r) \) are corresponding background metric structures such that $\tilde h$ is a $W^{1,p}_{-\tau}(h,r)$-asymptotically Euclidean metric and such that for some constant \( C \geq 1 \), we have $C^{-1}r \leq \tilde r \leq Cr$ on all of \( M \). Then
    \begin{equation*}
        L^\infty_0(h,r) \cap W^{1,p}_{-\tau}(h,r) \subseteq L^\infty_0(\tilde h,\tilde r) \cap W^{1,p}_{-\tau}(\tilde h,\tilde r).
    \end{equation*}
\end{proposition}

\begin{proof}
    Let \( T \in L^\infty_0(h,r) \cap W^{1,p}_{-\tau}(h,r) \). Since \(C^{-1}h \leq \tilde h \leq Ch \), it follows that \( \lvert T \rvert_{\tilde  h} \leq C \lvert T \rvert_{h} \) and that \( \det(\tilde h) \leq C\det(h) \). As \(C^{-1}r \leq \tilde r \leq Cr \) as well, when \( p < \infty \) we have the estimate
    \begin{equation*}
        \lVert T \rVert_{L^p_{-\tau}(\tilde h,\tilde r)}
        = \biggl(\int_M \tilde r^{p\tau - n}\lvert T\rvert_{\tilde h}^p \,d\mu_{\tilde h} \biggr)^{1/p}
        \leq C \biggl( \int_M r^{p\tau - n}\lvert T\rvert_{h}^p \,d\mu_{h} \biggr)^{1/p}
        = C\lVert T \rVert_{L^p_{-\tau}(h,r)}.
    \end{equation*} 
    By a similar calculation, \( \lVert T \rVert_{L^\infty_0(\tilde h,\tilde r)} \leq C\lVert T \rVert_{L^\infty_0(h,r)} \), so \( T \in L^\infty_0(\tilde h,\tilde r) \cap L^p_{-\tau}(\tilde h, \tilde r) \). Since \( \tilde h \) is a \( W^{1, p}_{-\tau}(h, r) \)-asymptotically Euclidean metric, we can apply Lemma \ref{lem:ChristofferSymbolsBounds} to the difference tensor \( \Gamma_{ij}^k \coloneq {}^{\tilde h}\Gamma_{ij}^k - {}^{h}\Gamma_{ij}^k \) to conclude that \( \Gamma \in L^p_{-\tau-1}(h, r) \). Due to the inclusions
    \begin{equation*}
        D T \in L^p_{-\tau-1}(h, r), \quad 
        \nabla^{\tilde  h}T - D T 
        = \underbrace{\Gamma}_{L^p_{-\tau-1}(h,r)} \star \underbrace{T}_{L^\infty_0(h,r)}
        \in L^p_{-\tau-1}(h,r),
    \end{equation*}
    and the calculation in the first part of the proof, we conclude that $\nabla^{\tilde h}T \in L^p_{-\tau-1}(\tilde h,\tilde r)$ as well. In conclusion we find that \( T \in L^\infty_0(\tilde h,\tilde r) \cap W^{1,p}_{-\tau}(\tilde h,\tilde r) \). 
\end{proof}

\begin{proposition}
    \label{prop:MassAndIsometries}
    Suppose that \( (M,\Phi) \) and \( (M, G \circ \Phi) \) are reference manifolds equipped with background metric structures \( (h,r) \) and \( (\tilde h, \tilde {r}) \), where \( G \colon \mathbb R^n \to \mathbb R^n \) is an isometry. Then the following holds.
    \begin{enumerate}
        \item Outside a compact subset of \( M \) we have \( h = \tilde h \), and for a constant \( C \geq 1 \) we have \( C^{-1}r < \tilde r < Cr \) on all of \( M \).
        \item If \( p \in [1,\infty] \), \( \tau \in \mathbb R \) and \( g \) is a \( W^{1,p}_{-\tau}(h,r) \)-asymptotically Euclidean metric on $M$, then the metric $g$ is \( W^{1,p}_{-\tau}(\tilde h, \tilde r) \)-asymptotically Euclidean as well.
        \item Let \( \{\chi_\alpha\}_{\alpha \geq 1} \) be a sequence of cutoff functions with respect to $(h,r)$, as in Definition \ref{def:WeakMass}. If the weak mass \( m_\textnormal{W}(g,\Phi,\{\chi_\alpha\}_{\alpha \geq 1}) \) is well-defined and independent of the choice of cutoff functions, then the weak mass \( m_\textnormal{W}(g,G \circ \Phi,\{\chi_\alpha\}_{\alpha \geq 1}) \) is also well-defined, independent of choice of cutoff functions and satisfies
        \begin{equation*}
            m_\textnormal{W}(g, G \circ \Phi) 
            = m_\textnormal{W}(g, \Phi).
        \end{equation*}
    \end{enumerate}
\end{proposition}

Note that we do not claim that \( (M, g) \) has distributional scalar curvature in \( L^1 \) in either chart or that $\tau \geq \frac{n-2}{2}$, as needed in order to apply Theorem \ref{thm:WeakMassWellDefined}.

\begin{proof}
    Since \( (M,\Phi) \) is a reference manifold there is a compact set \( K \subseteq M \) and a radius \( R > 0 \) such that \( \Phi(M \setminus K) = \mathbb R^n \setminus B_R \). Similarly since \( (M,G \circ \Phi) \) is a reference manifold there is a compact set \( K' \subseteq M \) and a radius \( R' > 0 \) such that \( (G \circ \Phi)(M \setminus K') = \mathbb R^n \setminus B_{R'} \) and without loss of generality we may assume that on \( M \setminus K' \) we have
    \begin{equation*}
        h = \Phi^*\delta, \quad r = \lvert x \rvert \circ \Phi, \quad \tilde h = (G \circ \Phi)^*\delta, \quad \tilde r = \lvert x \rvert \circ G \circ \Phi. 
    \end{equation*}
    Thus on \( M \setminus K' \) we have
    \begin{equation*}
        h 
        = \Phi^*\delta 
        = \Phi^*(G^*\delta) 
        = (G\circ \Phi)^* \delta 
        = \tilde h,
    \end{equation*}
    where in the second equality we use that \( G \) is an isometry with respect to \( \delta \). Thus \( h = \tilde h \) outside of the compact set \( K' \). Since any isometry of \( \mathbb R^n \) can be written as an orthogonal transformation followed by a translation we find
    \begin{equation*}
        \lim_{p \to \infty} \frac{\tilde r(p)}{r(p)} 
        = \lim_{x \to \infty} \frac{\lvert G(x) \rvert}{\lvert x \rvert} 
        = 1.
    \end{equation*}
    Thus there exists a compact set \( E \subseteq M \) outside of which we have \( r/2 < \tilde r < 2r \). Furthermore, since the functions \( r \colon M \to \mathbb R \) and \( \tilde r \colon M \to \mathbb R \) are positive and the set \( E \) is compact, there is a constant \( C > 1 \) such that \( C^{-1}r < \tilde r < Cr \) on \( E \). Thus \( C^{-1}r < \tilde r < Cr \) on all of \( M \) and we have proven the first part of the proposition. 

    Let \( p \in [1,\infty] \) and \( \tau \in \mathbb R \) and suppose that \( g \) is a \( W^{1,p}_{-\tau}(h,r) \)-asymptotically Euclidean metric on $M$. Since \( h = \tilde h \) outside of a compact set and \( C^{-1} r \leq \tilde r \leq Cr \) on all of \( M \), the metric \( \tilde h \) is $W^{1,p}_{-\tau}(h,r)$-asymptotically Euclidean. Moreover, \( g-h \in W^{1,p}_{-\tau}(h,r) \cap L^\infty_{0}(h,r) \) so Proposition \ref{prop:ComparingBMS} implies that \( g - h \in W^{1,p}_{-\tau}(\tilde h,\tilde r) \), and \( h - \tilde h \) is compactly supported so \( h - \tilde h \in W^{1,p}_{-\tau}(\tilde h, \tilde r) \). Thus
    \begin{equation*}
        g - \tilde h 
        = (g - h) + (h - \tilde h) 
        \in W^{1,p}_{-\tau}(\tilde h, \tilde r). 
    \end{equation*}
    This inclusion together with the inequalities
    \begin{equation*}
        C^{-2}\tilde h \leq C^{-1} h \leq g \leq Ch \leq C^2 \tilde h,
    \end{equation*}
    implies that $g$ is a $W^{1,p}_{-\tau}(\tilde h, \tilde r)$-asymptotically Euclidean metric, which is the second part of the proposition.
    
    Finally, if \( m_\text{W}(g,\Phi, \{\chi_\alpha\}_{\alpha \geq 1}) \) is well-defined and independent of cutoff functions \( \{\chi_\alpha\}_{\alpha \geq 1} \) as in Definition \ref{def:WeakMass}, we have
    \begin{equation*}
        \begin{aligned}
            m_\text{W}(g,G \circ \Phi, \{\chi_\alpha\}_{\alpha \geq 1})
            &= \frac{1}{2(n-1)\omega_{n-1}} \lim_{\alpha \to \infty}\int_M (\tilde h^{ij}\tilde h^{kl} - \tilde h^{ik} \tilde h^{jl})\nabla^{\tilde h}_k(g_{jl} - \tilde h_{jl})(-D_i\chi_\alpha) \,d\mu_{\tilde h}
            \\&= \frac{1}{2(n-1)\omega_{n-1}} \lim_{\alpha \to \infty}\int_M ( h^{ij}h^{kl} - h^{ik} h^{jl})D_k(g_{jl} - h_{jl})(-D_i\chi_\alpha) \,d\mu_h
            \\&= m_\text{W}(g,\Phi),
        \end{aligned}
    \end{equation*}
    where we recall that \( \nabla^{\tilde h} \) is the Levi-Civita connection with respect to \( \tilde h \). Here we used the fact that the supports of the \( D\chi_\alpha \) are contained in \( M \setminus K' \) for large enough \( \alpha \), at which point \( h = \tilde h \) and \( \nabla^{\tilde h} = \nabla^h = D \). 
\end{proof}

We now show that composing a chart at infinity with an ``almost identity'' does not change the mass. Due to the length of the argument we split the proof into two parts. Recall that we are abusing notation slightly, see Remark \ref{rem:AbuseOfNotation}.

\begin{proposition}
    \label{prop:CoordinatesAlmostUnity}
    Suppose that \( (M,\Phi) \) and \( (M, F \circ \Phi) \) are reference manifolds equipped with background metric structures \( (h,r) \) and \( (\tilde h, \tilde {r}) \), where \( F \colon \mathbb R^n \setminus B_R \to \mathbb R^n \) is a diffeomorphism for some \( R > 0 \), such that for some \( p \in (n,\infty] \) and \( \tau > 0\) we have \( \nabla^\delta ( F^i - x^i) \in {}^\delta W^{1,p}_{-\tau}(h,r) \) for all $i = 1,\dots, n$. Then \( \tilde h \) is a \( W^{1,p}_{-\tau}(h,r) \)-asymptotically Euclidean metric and for a constant \( C \geq 1 \) we have \( C^{-1} r \leq \tilde r \leq Cr \). Moreover, if \( g \) is a \( W^{1,p}_{-\tau}(h,r) \)-asymptotically Euclidean metric, then \( g \) is also a \( W^{1,p}_{-\tau}(\tilde h,\tilde r) \)-asymptotically Euclidean metric.
\end{proposition}

\begin{proof}  
    Lemma \ref{lem:WeightedSobolev} implies that for each index \( i \) we have \( \nabla^\delta ( F^i - x^i ) \in {}^\delta L^{\infty}_{-\tau} \). Fixing any \( x_0 \in \mathbb R^n \setminus B_R \), this in turn implies that \( \lvert (F(x) - x) - (F(x_0) - x_0) \rvert_\delta \leq C\lvert x \rvert^{1-\tau} \) for any \( x \in \mathbb R^n \setminus B_R \) and hence
    \begin{equation*}
        \lim_{p \to \infty} \frac{\tilde r(p)}{r(p)} 
        = \lim_{x \to \infty} \frac{\lvert F(x) \rvert}{\lvert x \rvert} 
        = 1.
    \end{equation*}
    Thus there exists a compact set \( E \subseteq M \) outside of which we have \( r/2 < \tilde r < 2r \). Since the functions \( r \colon M \to \mathbb R \) and \( \tilde r \colon M \to \mathbb R \) are positive and the set \( E \) is compact, there is a constant \( C > 1 \) such that \( C^{-1}r < \tilde r < Cr \) on \( E \), thus \( C^{-1}r < \tilde r < Cr \) on all of \( M \). 
    
    We now make the following observation:
    \begin{equation*}
        \begin{aligned}
            (F^* \delta)_{ij}
            = (D_iF^u)(D_jF^v)\delta_{uv}
            = \underbrace{(D_iF^u - \delta_i^u) \vphantom{\delta_j^v}}_{{}^\delta L^p_{-\tau}}
             \underbrace{(D_jF^v)  \delta_{uv} \vphantom{\delta_j^v}}_{{}^\delta L^\infty_0} 
            + \underbrace{(D_jF^v - \delta_j^v)}_{{}^\delta L^p_{-\tau}} \underbrace{\delta_i^u\delta_{uv} \vphantom{\delta_j^v}}_{{}^\delta L^\infty_0} 
            + \delta_i^u\delta_j^v \delta_{uv},
        \end{aligned}
    \end{equation*}
    so that \( F^*\delta - \delta \in {}^\delta L^p_{-\tau} \). At the same time we have
    \begin{equation*}
        D_k(F^* \delta - \delta)_{ij} 
        = \underbrace{(D_kD_iF^u) \vphantom{D_j}}_{{}^\delta L^p_{-\tau-1}}
          \underbrace{(D_jF^v)\delta_{uv}}_{{}^\delta L^\infty_0} 
        + \underbrace{(D_iF^u) \vphantom{D_j}}_{{}^\delta L^\infty_0}
          \underbrace{(D_k D_jF^v)}_{{}^\delta L^p_{-\tau-1}}
          \underbrace{\delta_{uv} \vphantom{D_j}}_{{}^\delta L^\infty_0} 
        \in {}^\delta L^p_{-\tau-1}.
    \end{equation*}
    Thus \( F^*\delta - \delta \in {}^\delta W^{1,p}_{-\tau} \). An application of Lemma \ref{lem:WeightedSobolev} implies that $ F^*\delta - \delta \in {}^\delta L^\infty_{-\tau}$ and hence
    \begin{equation*}
        \lvert F^* \delta - \delta \rvert_\delta 
        < C \lvert x \rvert^{-\tau}.
    \end{equation*}
    It follows that $C^{-1} \delta \leq F^* \delta \leq C \delta$ and due to the equality \( \Phi^* \delta = h \) we find
    \begin{equation*}
        \begin{split}
        \tilde h - h 
            & = (F \circ \Phi)^* \delta - \Phi^* \delta 
            = \Phi^*(F^* \delta - \delta) 
            \in W^{1,p}_{-\tau}(h,r), \\ 
            h & = \Phi^* \delta \leq C\Phi^*(F^* \delta) = C(F\circ \Phi)^* \delta = C\tilde h, \\
            \tilde h & = \Phi^* F^*\delta \leq C\Phi^*\delta = Ch.
        \end{split}
    \end{equation*}
    In the first line we recalled Remark \ref{rem:EuclideanSobolevs}, stating that a tensor \( T \in W^{1,p}_{-\tau}(h,r) \) if and only if \( \Phi_*T \in {}^\delta W^{1,p}_{-\tau} \) and in the second and third line we use that $B_1 \leq CB_2$ as bilinear forms if and only if $\Phi^* B_1 \leq C\Phi^* B_2$ as bilinear forms. The above three inequalities imply that $\tilde h$ is a $W^{1,p}_{-\tau}(h,r)$-asymptotically Euclidean metric. 

    Finally, let \( g \) be some \( W^{1,p}_{-\tau}(h,r) \)-asymptotically Euclidean metric. Applying Proposition \ref{prop:ComparingBMS} we see that $g \in W^{1,p}_{-\tau}(\tilde h,\tilde r)$. Moreover, since
    \begin{equation*}
        C^{-2} \tilde h \leq C^{-1} h \leq g \leq Ch \leq C^2 \tilde h,
    \end{equation*}
    we find that $g$ is $W^{1,p}_{-\tau}(\tilde h, \tilde r)$-asymptotically Euclidean. 
\end{proof}

With Proposition \ref{prop:CoordinatesAlmostUnity} and Proposition \ref{prop:ComparingBMS} in hand we now show that composing a chart at infinity with an ``almost identity'' does not change the mass. 

\begin{proposition}
    \label{prop:AlmostUnityMass}
    Suppose that \( (M,\Phi) \) and \( (M, F \circ \Phi) \) are reference manifolds equipped with background metric structures \( (h,r) \) and \( (\tilde h, \tilde {r}) \), where \( F \colon \mathbb R^n \setminus B_R \to \mathbb R^n \) is a diffeomorphism for some \( R > 0 \), such that for some real numbers \( n < p < \infty \), \( \tau > \frac{n-2}{2} \) we have for all \( i = 1,\dots, n \) 
    \begin{equation*}
        \nabla^\delta ( F^i - x^i ) \in {}^\delta W^{1,p}_{-\tau}, \quad 
        \nabla^\delta \bigl((F^{-1})^i - x^i \bigr) \in {}^\delta W^{1,p}_{-\tau}.
    \end{equation*}
    Let \( \{\chi_\alpha\}_{\alpha \geq 1} \) be a sequence of cutoff functions with respect to \( (h,r) \), as in Definition \ref{def:WeakMass}. If the weak mass \( m_\textnormal{W}(g,\Phi,\{\chi_\alpha\}_{\alpha \geq 1}) \) is well-defined and independent of the choice of cutoff functions, then \( m_\textnormal{W}(g,F \circ \Phi,\{\chi_\alpha\}_{\alpha \geq 1}) \) is also well-defined, independent of the choice of cutoff functions and satisfies
    \begin{equation*}
        m_\textnormal{W}(g, F \circ \Phi)
        = m_\textnormal{W}(g,\Phi).
    \end{equation*}
\end{proposition}

Once again, we do not claim that \( (M, g) \) has distributional scalar curvature in \( L^1 \) in either chart. We merely show that as soon as \( m_\text{W}(g,\Phi) \) exists, then \( m_\text{W}(g,F\circ \Phi) \) also exists and the two masses agree.

\begin{proof}
    Since \( (M,\Phi) \) and \( (M,F \circ \Phi) \) are reference manifolds there are compact sets \( K, K' \subseteq M \) and radii \( R, R' > 0 \) such that \( \Phi(M \setminus K) = \mathbb R^n \setminus B_R \) and \( (F \circ \Phi)(M \setminus K') = \mathbb R^n \setminus B_{R'} \). Without loss of generality we may assume that \( K \supset K' \) and \( R > R' \), so that on \( M \setminus K \) we have
    \begin{equation*}
        h = \Phi^*\delta, \quad 
        r = \lvert x \rvert \circ \Phi, \quad 
        \tilde h = (F\circ \Phi)^*\delta, \quad 
        \tilde r = \lvert x \rvert \circ F \circ \Phi. 
    \end{equation*} 
    Let \( \{ \widetilde \chi_\alpha\}_{\alpha \geq 1} \) be a sequence of cutoff functions with respect to the background metric structure \( (\tilde h, \tilde r) \), as in Definition \ref{def:WeakMass}. We choose the index \( \alpha_0 \) so that we have \( \widetilde \chi_\alpha^{-1}(1) \supset K \) for all \( \alpha \geq \alpha_0 \) and define two more sequences of cutoff functions, \( \{\chi_\alpha\}_{\alpha \geq \alpha_0} \) and \( \{\overline \chi_\alpha\}_{\alpha \geq \alpha_0} \), which we make use of in the proof. We define
    \begin{align*}
        \chi_\alpha 
        & \coloneq 
        \begin{cases}
            \widetilde \chi_\alpha \circ (F \circ \Phi)^{-1} \circ \Phi, & \text { on } M \setminus K \\
            1 , & \text{ on } K
        \end{cases}, \\
        \overline \chi_\alpha 
        & \coloneq
        \begin{cases}
            \chi_\alpha \circ\Phi^{-1}, & \text{ on } \mathbb R^n \setminus B_R \\
            1, & \text{ on } B_R
        \end{cases}
        ,
    \end{align*}
    and note that each \( \chi_\alpha \colon M \to \mathbb R \) and \( \overline \chi_\alpha \colon \mathbb R^n \to \mathbb R \) is locally Lipschitz and compactly supported. By definition we have
    \begin{equation*}
        \bigcup_{\alpha \geq \alpha_0} \chi_\alpha^{-1}(1) = M, \quad \bigcup_{\alpha \geq \alpha_0} \overline \chi_\alpha^{-1}(1) = \mathbb R^n.
    \end{equation*}
    Since the functions \( \chi_\alpha \) and \( \overline \chi_\alpha \) are identically equal to \( 1 \) on the compact set \( K \) respectively \( B_R \), we have the equalities
    \begin{align*}
        \lVert \chi_\alpha \rVert_{W^{1,\infty}_0(h,r)}
        & = \lVert \widetilde \chi_\alpha \circ (F \circ \Phi)^{-1} \rVert_{{}^\delta W^{1,\infty}_0}
        = \lVert \widetilde \chi_\alpha \rVert_{W^{1,\infty}_0(\tilde h,\tilde r)}
        < \infty, \\
        \lVert \overline \chi_\alpha \rVert_{{}^\delta W^{1,\infty}_0}
        & = \lVert \chi_\alpha \circ \Phi^{-1} \rVert_{{}^\delta W^{1,\infty}_0}
        = \lVert \widetilde \chi_\alpha \circ (F \circ \Phi)^{-1} \rVert_{{}^\delta W^{1,\infty}_0}
        = \lVert \widetilde \chi_\alpha \rVert_{W^{1,\infty}_0(\tilde h,\tilde r)} 
        < \infty.
    \end{align*}
    In particular,
    \begin{equation}
        \label{eq:EverythingCutoff}
        \sup_\alpha \lVert \chi_\alpha \rVert_{W^{1,\infty}_0(h,r)} 
        = \sup_\alpha \lVert \overline \chi_\alpha \rVert_{{}^\delta W^{1,\infty}_0} 
        = \sup_\alpha \lVert \widetilde \chi_\alpha \rVert_{W^{1,\infty}_0(\tilde h,\tilde r)} 
        < \infty
    \end{equation}
    and hence \( \{ \chi_\alpha \}_{\alpha \geq \alpha_0} \) is a sequence of cutoff functions on \( M \) with respect to the background metric structure \( (h,r) \) as in Definition \ref{def:WeakMass}. 
    We now show that the difference 
    \begin{equation*}
        \begin{aligned}
            \mathcal D 
            \coloneq {}& 2\omega_{n-1}(n-1)\bigl( m_\text{W}(g,F \circ \Phi,\{\widetilde \chi_\alpha\}_{\alpha \geq \alpha_0}) - m_\text{W}(g, \Phi,\{\chi_\alpha\}_{\alpha \geq \alpha_0}) \bigr) 
            \\= {}& \lim_{\alpha \to \infty} \int_M (\tilde h^{ij}\tilde h^{kl} - \tilde h^{ik}\tilde h^{jl})\nabla^{\tilde h}_kg_{jl}(-\nabla^{\tilde h}_i \widetilde \chi_\alpha) \,d\mu_{\tilde h} 
            - \int_M (h^{ij}h^{kl} - h^{ik} h^{jl})D_kg_{jl}(-D_i \chi_\alpha) \,d\mu_h,
        \end{aligned}
    \end{equation*}
    vanishes. Using a change of variables \( x = (F \circ \Phi)(p) \) in the first integral and \( x = \Phi(p) \) in the second, and recalling that \( \text{supp}(\nabla^{\tilde h} \widetilde \chi_\alpha) \) and \( \text{supp}(D \chi_\alpha) \) are contained in \( M \setminus K \) we find
    \begin{equation*}
        \mathcal D 
        =
        \begin{aligned}[t]
            &\lim_{\alpha \to \infty}\int_{\mathbb R^n} (\delta^{ij}\delta^{kl} - \delta^{ik}\delta^{jl})\nabla^\delta_k \bigl( (F \circ \Phi)_*g \bigr)_{jl} \bigl( -\partial_i(\widetilde \chi_\alpha \circ (F \circ \Phi)^{-1}) \bigr) \,d\mu_\delta 
            \\&- \lim_{\alpha \to \infty}\int_{\mathbb R^n} (\delta^{ij}\delta^{kl} - \delta^{ik} \delta^{jl})\nabla^\delta_k (\Phi_* g)_{jl} \bigl( -\partial_i(\chi_\alpha \circ \Phi^{-1}) \bigr) \,d\mu_\delta. 
        \end{aligned}
    \end{equation*}
    We now define 
    \begin{equation*}
        \bar g 
        \coloneq
        \begin{cases}
            \Phi_*g, & \text{ on } \mathbb R^n \setminus B_R \\
            \delta, & \text{ on } B_R
        \end{cases},
    \end{equation*}
    so that after recalling the definitions of \( \chi_\alpha \) and \( \bar \chi_\alpha \) we find that
    \begin{equation*}
        \mathcal D 
        = \lim_{\alpha \to \infty}\int_{\mathbb R^n} (\delta^{ij}\delta^{kl} - \delta^{ik}\delta^{jl})\nabla^\delta_k (F_*\bar g - \bar g)_{jl} ( -\partial_i \overline \chi_\alpha) \,d\mu_\delta.
    \end{equation*}
    Defining \( A \coloneq F^{-1}\) so that \( A^* = F_* \) and \( \bar e \coloneq \bar g - \delta \), the above can be written as follows:
    \begin{equation}
        \label{eq:MassDiff2}
        \mathcal D
        =
        \begin{aligned}[t]
            &-\lim_{\alpha \to \infty}\int_{\mathbb R^n} (\delta^{ij}\delta^{kl} - \delta^{ik}\delta^{jl}) \nabla^\delta_k (A^*\delta - \delta)_{jl} (-\partial_i \overline \chi_\alpha) \,d\mu_\delta \\
            & + \lim_{\alpha \to \infty}\int_{\mathbb R^n} (\delta^{ij}\delta^{kl} - \delta^{ik}\delta^{jl}) \nabla^\delta_k (A^*\bar e - \bar e )_{jl} (-\partial_i \overline \chi_\alpha) \,d\mu_\delta.
        \end{aligned}
    \end{equation}
    Next we note that
    \begin{equation*}
        (A^*\bar e - \bar e)_{jl}
        = \bigl( (\partial_j A^u)(\partial_l A^v) - \delta^u_i\delta^v_l \bigr) \bar e_{uv} 
        = (\partial_j A^u - \delta^u_j)(\partial_l A^v)\bar e_{uv} + \delta^u_j(\partial_l A^v - \delta^v_l)\bar e_{uv}.
    \end{equation*}
    By Remark \ref{rem:EuclideanSobolevs} we have \( \bar e \in {}^\delta W^{1,p}_{-\tau} \subseteq {}^\delta W^{1,2}_{1-n/2} \) and since \( g \) is an asymptotically Euclidean metric we have \( \bar e \in {}^\delta L^\infty_0\). By assumption we have \( \nabla^\delta ( A^i - x^i ) \in {}^\delta W^{1,p}_{-\tau} \subseteq {}^\delta W^{1,2}_{1-n/2} \) and due to Lemma \ref{lem:WeightedSobolev} it follows that \( \nabla^\delta ( A^i - x^i ) \in {}^\delta L^\infty_{-\tau} \), hence
    \begin{equation*}
        \nabla^\delta_k(A^*\bar e - \bar e)_{jl}
        =
        \begin{aligned}[t]
            &\underbrace{ (\partial_j A^u - \delta_j^u)                         }_{{}^\delta L^2_{1-n/2} }
            \underbrace{ (\partial_l A^v)                \vphantom{\delta_j^u} }_{{}^\delta L^\infty_0     }
            \underbrace{ \partial_k \bar e_{uv}          \vphantom{\delta_j^u} }_{{}^\delta L^2_{-n/2}     } 
            + \underbrace{ \delta_j^u              \vphantom{\delta_j^u} }_{{}^\delta L^\infty_0     }
              \underbrace{ (\partial_l A^v - \delta_l^v) \vphantom{\delta_j^u} }_{{}^\delta L^2_{1-n/2} }
              \underbrace{ \partial_k \bar e_{uv}        \vphantom{\delta_j^u} }_{{}^\delta L^2_{-n/2}     }
            \\&+ 
            \underbrace{ (\partial_j A^u - \delta_j^u)                         }_{{}^\delta L^2_{1-n/2} }
            \underbrace{ (\partial_k \partial_l A^v)     \vphantom{\delta_j^u} }_{{}^\delta L^2_{-n/2}     }
            \underbrace{ \bar e_{uv}                     \vphantom{\delta_j^u} }_{{}^\delta L^\infty_0     } 
            + \underbrace{ 0 \vphantom{\delta_j^u} }_{{}^\delta L^2_{-n/2} }
              \underbrace{ (\partial_l A^v - \delta_l^v) \vphantom{\delta_j^u} }_{{}^\delta L^2_{1-n/2} }
              \underbrace{ \bar e_{uv}                   \vphantom{\delta_j^u} }_{{}^\delta L^\infty_0     }
            \\&+ 
            \underbrace{ (\partial_k\partial_j A^u - 0)   \vphantom{\delta_j^u} }_{{}^\delta L^2_{-n/2}     }
            \underbrace{ (\partial_l A^v)             \vphantom{\delta_j^u} }_{{}^\delta L^\infty_0     }
            \underbrace{ \bar e_{uv}                  \vphantom{\delta_j^u} }_{{}^\delta L^2_{1-n/2} } 
            + \underbrace{ \delta_j^u           \vphantom{\delta_j^u} }_{{}^\delta L^\infty_0     }
              \underbrace{ (\partial_k\partial_l A^v - 0) \vphantom{\delta_j^u} }_{{}^\delta L^2_{-n/2}     }
              \underbrace{ \bar e_{uv}                \vphantom{\delta_j^u} }_{{}^\delta L^2_{1-n/2} } 
            \in {}^\delta L^1_{1-n}. 
        \end{aligned}
    \end{equation*}
    This inclusion together with the bound \eqref{eq:EverythingCutoff} allows us to apply Lemma \ref{lem:IntegralsToZero} to conclude that the second limit of \eqref{eq:MassDiff2} vanishes. Thus
    \begin{equation}
        \label{eq:MassDiff}
        \mathcal D 
        = \lim_{\alpha \to \infty}\int_{\mathbb R^n} (\delta^{ij}\delta^{kl} - \delta^{ik}\delta^{jl}) \nabla^\delta_k (A^*\delta)_{jl}(-\partial_i\overline \chi_\alpha) \,d\mu_\delta.
    \end{equation}
    The tensor \( A^*\delta \) satisfies
    \begin{equation*}
        (A^*\delta)_{jl} 
        = (\partial_j A^u)(\partial_l A^v)\delta_{uv}, \quad
        \nabla^\delta_k(A^*\delta)_{jl} 
         = \bigl( (\partial_k\partial_jA^u)(\partial_l A^v) + (\partial_jA^u)(\partial_k\partial_l A^v) \bigr) \delta_{uv}.
    \end{equation*}
    Hence we have the inclusion
    \begin{equation*}
        \begin{aligned}
            \nabla^\delta_k(A^*\delta)_{jl} - (&\delta_{ul} \partial_k \partial_j A^u  + \delta_{vj} \partial_k\partial_l A^v)
            \\&=\bigl( (\partial_k\partial_jA^u)(\partial_l A^v) + (\partial_jA^u)(\partial_k\partial_l A^v) \bigr) \delta_{uv} 
            - \bigl( (\partial_k\partial_jA^u)\delta_l^v + \delta_j^u(\partial_k\partial_l A^v) \bigr) \delta_{uv} 
            \\&= \bigl( 
            \underbrace{ (\partial_k \partial_j A^u)   \vphantom{\delta_j^u} }_{{}^\delta L^2_{-n/2}   } 
            \underbrace{ (\partial_l A^v - \delta_l^v) \vphantom{\delta_j^u} }_{{}^\delta L^2_{1-n/2}  } 
            + 
            \underbrace{ (\partial_jA^u - \delta_j^u)                        }_{{}^\delta L^2_{1-n/2}  }
            \underbrace{ (\partial_k\partial_l A^v)    \vphantom{\delta_j^u} }_{{}^\delta L^2_{-n/2}   } \bigr)
            \underbrace{ \delta_{uv}                   \vphantom{\delta_j^u} }_{{}^\delta L^\infty_{0} }
            \in {}^\delta L^1_{1-n}.
        \end{aligned}
    \end{equation*}
    The above combined with \( \delta^{ij}\delta^{kl} - \delta^{ik}\delta^{jl} \in {}^\delta L^\infty_0 \)  implies
    \begin{equation}
        \label{eq:ImportantInclusion}
        (\delta^{ij}\delta^{kl} - \delta^{ik}\delta^{jl})(\nabla^\delta_k(A^*\delta)_{jl} - (\delta_{ul} \partial_k \partial_j A^u  + \delta_{vj} \partial_k\partial_l A^v)) \in {}^\delta L^1_{1-n}.
    \end{equation}
    The second term in \eqref{eq:ImportantInclusion} can be rewritten as follows:
    \begin{align*}
         (\delta^{ij}\delta^{kl} - \delta^{ik}\delta^{jl} )
        (\delta_{ul} \partial_k \partial_j A^u  + \delta_{vj} \partial_k\partial_l A^v ) 
        &= (\delta^{ij}\delta^{kl} - \delta^{ik}\delta^{jl})
        \delta_{ul} \partial_k \partial_j A^u 
        + (\delta^{ij}\delta^{kl} - \delta^{ik}\delta^{jl})
        \delta_{vj} \partial_k\partial_l A^v \\
        & = (\delta^{ij}\delta^k_u - \delta^{ik}\delta^j_u)\partial_k \partial_j A^u 
        + (\delta^i_v\delta^{kl} - \delta^{ik}\delta^l_v)\partial_k\partial_l A^v \\
        & = \delta^{ij}\partial_k \partial_j A^k - \delta^{ik}\partial_k \partial_j A^j + \Delta^\delta A^i - \delta^{ik}\partial_k\partial_l A^l \\
        & = \Delta^\delta A^i - \delta^{ij} \partial_j \partial_k A^k,
    \end{align*}
    where \( \Delta^\delta A^i \coloneq \delta^{jk} \partial_j \partial_k A^i \) is the Laplacian of \( A^i \) with respect to \( \delta \). Substituting, \eqref{eq:ImportantInclusion} now reads
    \begin{equation*}
        (\delta^{ij}\delta^{kl} - \delta^{ik}\delta^{jl}) \nabla^\delta_k(A^*\delta)_{jl} - (\Delta^\delta A^i - \delta^{ij}\partial_j\partial_kA^k) 
        \in {}^\delta L^1_{1-n}.
    \end{equation*}
    Due to this inclusion and \eqref{eq:EverythingCutoff}, we can apply Lemma \ref{lem:IntegralsToZero} once more to conclude that \eqref{eq:MassDiff} can be written as
    \begin{equation*}
    \begin{aligned}
        \mathcal D
        & = 
        \begin{aligned}[t]
            &\lim_{\alpha \to \infty}\int_{\mathbb R^n} (\Delta^\delta A^i - \delta^{ij}\partial_j\partial_kA^k)(-\partial_i\bar\chi_\alpha) \,d\mu_\delta
            \\& + \lim_{\alpha \to \infty}\int_{\mathbb R^n} \bigl( (\delta^{ij}\delta^{kl} - \delta^{ik}\delta^{jl}) \nabla^\delta_k(A^*\delta)_{jl} - (\Delta^\delta A^i - \delta^{ij}\partial_j\partial_kA^k) \bigr) (-\partial_i\overline \chi_\alpha) \,d\mu_\delta
        \end{aligned}
        \\& = \lim_{\alpha \to \infty}\int_{\mathbb R^n} (\Delta^\delta A^i - \delta^{ij}\partial_j\partial_kA^k)(-\partial_i\bar\chi_\alpha) \,d\mu_\delta.
    \end{aligned}
    \end{equation*}
    We now show that each integral
    \begin{equation*}
        \mathrm I_\alpha 
        \coloneq \int_{\mathbb R^n} (\Delta^\delta A^i - \delta^{ij}\partial_j\partial_kA^k)(-\partial_i\overline \chi_\alpha) \,d\mu_\delta
    \end{equation*}
    vanishes. Since \( (F \circ \Phi)(M \setminus K') = \mathbb R^n \setminus B_{R'} \), we know that \( A \) is defined outside of a compact subset of \( \mathbb R^n \). We can thus let \( \phi \colon \mathbb R^n \to [0,1] \) be a compactly supported smooth function that equals \( 1 \) on a large enough compact set and the map \( \phi \text{id} + (1-\phi) A \) is then defined on all of \( \mathbb R^n \). Replacing the map \( A \) in the definition of \( \mathrm I_\alpha \) with \( \phi \text{id} + (1-\phi) A \) does not change \( \mathcal D \), since \( \phi \text{id} + (1-\phi)A = A \) outside of a compact set. Thus we may without loss of generality assume that \( A \colon \mathbb R^n \to \mathbb R^n \). For the moment we assume that each \( A^i \in C^3_\text{loc}(\mathbb R^n \setminus B_R) \). In this case we may define the vector field \( X \in C^1_\text{loc}(\mathbb R^n \setminus B_R) \) by \( X^i \coloneq \Delta^\delta A^i - \delta^{ij} \partial_j \partial_k A^k \) and calculate that
    \begin{equation*}
        \div_\delta (X) 
        = \partial_i(\Delta^\delta A^i - \delta^{ij}\partial_j\partial_kA^k)
        = \partial_i(\delta^{jk} \partial_j \partial_k A^i - \delta^{ij}\partial_j\partial_kA^k)
        = \delta^{jk} \partial_i\partial_j \partial_k A^i - \delta^{ij}\partial_k\partial_i\partial_jA^k 
        = 0,
    \end{equation*}
    where in the last line we have used that higher order partial derivatives with respect to $\delta$ commute. The divergence theorem implies
    \begin{equation*}
        \mathrm I_\alpha
        = -\int_{\mathbb R^n} \langle X, \nabla^\delta  \overline\chi_\alpha\rangle_\delta \,d\mu_\delta
        = -\int_{\mathbb R^n} \overline \chi_\alpha \div_\delta(X) \,d\mu_\delta
        = 0.
    \end{equation*}
    A standard approximation argument shows that \( \mathrm I_\alpha = 0 \) even when we only have \( A^i \in {}^\delta W^{2,p}_\text{loc} \). Thus \( \mathcal D = 0 \) and hence 
    \begin{equation*}
        m_\text{W} \bigl( g,F \circ \Phi,\{ \widetilde \chi_\alpha\}_{\alpha \geq 1} \bigr) = m_\text{W} \bigl( g,\Phi,\{ \chi_\alpha\}_{\alpha \geq 1} \bigr),
    \end{equation*}
    where \(  \chi_\alpha = \widetilde \chi_\alpha \circ (F \circ \Phi)^{-1} \circ \Phi \). Since the right hand side is independent of the sequence \( \{\chi_\alpha\}_{\alpha \geq 1} \), the left hand side is independent of the chosen sequence \( \{\widetilde \chi_\alpha\}_{\alpha \geq 1} \). 
\end{proof}

We conclude this section with the following theorem, which tells us that the ADM mass of a \( W^{1,p}_{-\tau}(h,r) \)-asymptotically Euclidean metric with \( p > n \) and \( \tau > \frac{n-2}{2} \) does not depend on the background metric structure \( (h,r) \). We note that in this case \( W^{1,p}_{-\tau} \subseteq W^{1,2}_{-(n-2)/2} \) and that in the conditions described below, Theorem \ref{thm:WeakMassWellDefined} implies that \( m_\textnormal{W}(g,\Phi,\{\chi_\alpha\}_{\alpha \geq 1}) \) is well-defined and independent of the choice of cutoff functions, cf.\ Bartnik \cite[Theorem $4.2$]{Bartnik} and Chruściel \cite[Theorem $2$]{Chrusciel}.

\begin{theorem} 
    \label{thm:CoordinateInvariance}
    Suppose that \( p > n \) and \( \tau > \frac{n-2}{2} \) are real numbers, that \( (M,\Phi) \) respectively \( (M,\widetilde \Phi) \) are reference manifolds equipped with background metric structures \( (h,r) \) respectively \( (\tilde h,\tilde r) \) and that \( g \) is both a \( W^{1,p}_{-\tau}(h,r) \) and a \( W^{1,p}_{-\tau}(\tilde h,\tilde r) \)-asymptotically Euclidean metric. If \( m_\textnormal{W}(g,\Phi,\{\chi_\alpha\}_{\alpha \geq 1}) \) is well-defined, independent of the choice of cutoff functions then so is \( m_\textnormal{W}(g,\widetilde \Phi,\{\chi_\alpha\}_{\alpha \geq 1}) \) and
    \begin{equation*}
        m_\textnormal{W}(g,\widetilde \Phi) 
        = m_\textnormal{W}(g,\Phi).
    \end{equation*}
\end{theorem}

\begin{proof}
    Let \( G \) be the isometry given by Proposition \ref{prop:UniquenessOfInfinity}. Then for \( F \coloneq G \circ \widetilde \Phi \circ \Phi^{-1} \) we have
    \begin{equation*}
        \nabla^\delta ( F^i - x^i ) \in {}^\delta W^{1,p}_{-\tau}, \quad 
        \nabla^\delta \bigl( (F^{-1})^i - x^i \bigr) \in {}^\delta W^{1,p}_{-\tau}.
    \end{equation*}
    We now have three background manifolds equipped with background metric structures, namely $(M,\Phi)$ equipped with $(h,r)$, $(M,\widetilde \Phi)$ equipped with $(\tilde h,\tilde r)$ and $(M, F \circ \Phi) = (M,G \circ \widetilde \Phi)$ equipped with $(\tilde h,\tilde r)$. By assumption, the mass \( m_\text{W}(g,\Phi, \{\chi_\alpha\}_{\alpha \geq 1}) \) is well-defined and independent of the choice of cutoff functions. An application of Proposition \ref{prop:AlmostUnityMass} implies that the mass $m_\text{W}(g,F\circ \Phi, \{\chi_\alpha\}_{\alpha \geq 1}) = m_\text{W}(g,G \circ \widetilde \Phi, \{\chi_\alpha\}_{\alpha \geq 1})$ is also well-defined, independent of the choice of cutoff functions and satisfies
    \begin{equation*}
        m_\text{W}(g, \Phi) 
        = m_\text{W}(g,G \circ \widetilde \Phi).
    \end{equation*}
    Applying Proposition \ref{prop:MassAndIsometries} with \( (M,G \circ \widetilde \Phi) \) in place of \( (M,\Phi) \) and \( G^{-1} \) in place of \( G \) shows that \( m_\text{W}(g,G^{-1} \circ G \circ \widetilde \Phi, \{\chi_\alpha\}_{\alpha \geq 1}) \) is well-defined, independent of cutoff functions, and satisfies
    \begin{equation*}
         m_\text{W}(g,\widetilde \Phi) 
         = m_\text{W}(g,G^{-1} \circ G \circ \widetilde \Phi) 
         = m_\text{W}(g,G \circ \widetilde \Phi).
    \end{equation*}
    All in all, we obtain \( m_\text{W}(g,\widetilde \Phi) = m_\text{W}(g, \Phi) \) as desired.
\end{proof}

%% file: sec_RicciMass.tex
\section{Ricci style definitions in low regularity}
\label{sec:RicciMass}

We recall that the Einstein tensor of a Riemannian manifold \( (M,g) \) is defined as
\begin{equation*}
    G^g
    \coloneq \Ric^g - \frac{\Scal^g}{2}g.
\end{equation*}
We also recall the Ricci version of the ADM mass, cf.\ Herzlich \cite[Definition \( 1.4 \)]{Herz}.

\begin{definition}
    \label{def:RicciVersions}
    Let \( g \) be a \( C^2 \)-asymptotically Euclidean metric of order \( \tau > 0 \) in the sense of Definition \ref{def:SmoothAE} and let \( X \) be the radial vector field \( X \coloneq \Phi^*(r\partial_r) \). The \emph{Ricci version of the ADM mass} of \( g \) is defined as
    \begin{equation*}
        m_\text{R}(g,\Phi) 
        \coloneq \frac{-1}{(n-1)(n-2)\omega_{n-1}} \lim_{R \to \infty} \int_{\mathbb S_R} G^g(X,\nu_h) \,d\mu_h,
    \end{equation*}
    where \( \nu_h \) is the outward pointing unit normal to \( \mathbb S_R = \{p \in M: r(p) = R\} \) with respect to \( h \), whenever the above limit exists.
\end{definition}

In \cite[Theorem $2.3$]{Herz} it is shown that for a \( C^2 \)-asymptotically Euclidean metric \( g \) of order \( \tau > \frac{n-2}{2} \) which has integrable scalar curvature and additionally satisfies \( g \in C^3_\text{loc} \), the Ricci version of the ADM mass of \( g \) is well-defined and agrees with the classical ADM mass as in Definition \ref{def:MassSmooth}. Motivated by this result, we show that the weak ADM mass of an \( W^{2,2}_{-(n-2)/2} \)-asymptotically Euclidean metric \( g \) can be expressed in terms of the Einstein tensor \( G^g \). We begin by making the following definition, analogous to Definition \ref{def:WeakMass}. 

\begin{definition}
    \label{def:WeakRicciVersions} 
    Suppose that \( g \) is a \( W^{2,2}_{-\tau} \)-asymptotically Euclidean metric for some \( \tau \in \mathbb R \) as in Definition \ref{def:AESobolev} and let \( \{\chi_\alpha\}_{\alpha \geq 1} \) be a sequence of cutoff functions as in Definition \ref{def:WeakMass}. Then the \emph{Ricci version of the weak ADM mass} of \( g \) is defined as
    \begin{equation*}
        \begin{split}
            m_\textnormal{RW} \bigl( g,\Phi,\{\chi_\alpha\}_{\alpha \geq 1} \bigr) 
            \coloneq{} & \frac{-1}{(n-1)(n-2)\omega_{n-1}} \lim_{\alpha \to \infty} \int_M G^g(rDr,-D\chi_\alpha) \,d\mu_h \\
            ={} & \frac{-1}{(n-1)(n-2)\omega_{n-1}} \lim_{\alpha \to \infty} \int_M h^{ik}h^{jl}G_{kl}^g (rD_ir)(-D_j\chi_\alpha) \,d\mu_h
        \end{split}
    \end{equation*}
    whenever this limit exists. Here we have abused notation slightly by denoting the gradients of \( \chi_\alpha \) and \( r \) with respect to the metric \( h \) by \( D\chi_\alpha \) and \( Dr \).
\end{definition}

The following technical lemma follows from Proposition \ref{prop:CubicIntegrability}, see also \cite[Lemma $4.8$]{GicSak}.

\begin{lemma}
    \label{lem:TechnicalWithGagliardoNirenberg}
    Let \( g \) be a \( W^{2,2}_{-\tau} \)-asymptotically Euclidean metric and let \( X \in L^2_{w-\tau} \cap L^3_{w-2\tau/3} \) for some \( \tau, w \in \mathbb R \). Then \(  G^g \star X \in L^1_{w-2-2\tau} \).
\end{lemma}

\begin{proof}
    Since \( \Scal^g = g^{-1} \star \Ric^g \) and \( g^{-1} \in L^\infty_0 \), it suffices to show that \( {\Ric^g} \star X \in L^1_{w-2-2\tau} \). To this end we recall Proposition \ref{prop:FirstOrderRicci}, which states that
    \begin{equation*}
        \Ric^g 
        = g^{-1} \star DDe + {\Ric^h} + g^{-1} \star g^{-1} \star De \star De.
    \end{equation*}
    When \( 3 \leq n \leq 6 \), Proposition \ref{prop:CubicIntegrability} implies
    \begin{equation*}
        \lVert De \rVert_{L^3_{-1-2\tau/3}} 
        \leq C\lVert De \rVert_{L^3_{-1-\tau}} 
        \leq C \lVert De \rVert_{W^{1,2}_{-1-\tau}} 
        \leq C \lVert e \rVert_{W^{2,2}_{-\tau}} 
        < \infty,
    \end{equation*}
    and when \( 7 \leq n \), Proposition \ref{prop:CubicIntegrability} implies
    \begin{equation*}
        \lVert De \rVert_{L^3_{-1-2\tau/3}} 
        \leq C \bigl(\lVert e \rVert_{W^{2,2}_{-1-\tau}} + \lVert e \rVert_{L^\infty_0} \bigr)   
        < \infty.
    \end{equation*}
    All in all we have that \( De \in L^3_{-1-2\tau/3} \), for all \( n \geq 3\). All in all
    \begin{equation*}
        {\Ric^g} \star X
        =     \underbrace{ (\Ric^h + g^{-1} \star DDe)      }_{ L^2_{-\tau-2}        } 
        \star \underbrace{ X                   \vphantom{)} }_{ L^2_{w-\tau}         } 
        +     \underbrace{ g^{-1} \star g^{-1} \vphantom{)} }_{ L^\infty_0           } 
        \star \underbrace{ De \star De         \vphantom{)} }_{ L^{3/2}_{-2-4\tau/3} } 
        \star \underbrace{ X                   \vphantom{)} }_{ L^3_{w-2\tau/3}      }
        \in L^1_{w-2-2\tau},
    \end{equation*}
    which is what we wanted to show. 
\end{proof}

The following proposition plays a similar role as Definition \ref{def:WeakScalarCurvature}. See also \cite[Section \( 4.2 \), Equation \(4.24\)]{GicSak}. We recall that the difference tensor \( \Gamma \) is given by \( \Gamma = {}^g \nabla - D \), see also \eqref{eq:DifferenceTensorComponents}.

\begin{proposition}
    \label{prop:ConformalKillingEinstein}
    Let \( g \in W^{2,2}_\textnormal{loc} \) be a metric and let \( X \) be a conformal Killing vector field for the metric \( h \) on the open set \( U \subseteq M \). Then for all locally Lipschitz functions \( \phi \colon M \to \mathbb R \) with compact support in \( U \) we have
    \begin{equation*}
        \int_M G^g(X,\nabla^g\phi) \,d\mu_g
        = -\int_M \phi g^{ij} g^{kl} G^g_{ik} \biggl( \frac{\div_h(X)}{n}h_{jl} + \Gamma^u_{jl}g_{uv} X^v \biggr) \,d\mu_g.
    \end{equation*}
\end{proposition}

\begin{proof}
    Without loss of generality, we can assume that \( g \) is smooth. The general case follows from a standard approximation argument, cf.\ the proof of \cite[Equation $4.18$]{GicSak}. We define a one-form \( w \coloneq \phi G^g(X,\cdot \,) \) and calculate as follows:
    \begin{equation*}
        \begin{aligned}
            \div_g(w)
            &= g^{ij} \bigl((\nabla^g_j\phi) G^g_{ki}X^k  + \phi (\nabla^g_j G^g_{ki})X^k + \phi G^g_{ki}(\nabla^g_j X^k) \bigr)
            \\&= g^{ij} \bigl((\nabla^g_j\phi) G^g_{ki}X^k + \phi G^g_{ki}(\nabla^g_j X^k) \bigr)
            \\&= G^g(X,\nabla^g\phi) + g^{ij}g^{kl}\phi G^g_{ki}(\nabla^g_j X_l),
        \end{aligned}
    \end{equation*}
    where in the second equality we have used that \( G^g \) is divergence free and in the third equality we have lowered one of the indices of \( \nabla^g X \) using \( g \). Integrating the above over \( M \) with respect to \( d\mu_g \) and noting that \( w \) is locally Lipschitz and compactly supported in \( U \), we conclude after an application of the divergence theorem that
    \begin{equation*}
        0 = \int_M G^g(X,\nabla^g\phi) + \phi g^{ij}g^{kl} G^g_{ki} (\nabla^g_j X_l) \,d\mu_g.
    \end{equation*}
    Since \( G^g \) is symmetric and since \( g^{ij} g^{kl} = g^{kl} g^{ij} \), we find that
    \begin{equation}
        \label{eq:AlmostConformal}
        \int_M G^g(X,\nabla^g\phi) \,d\mu_g
        = -\int_M \phi g^{ij}g^{kl}G^g_{ik}\frac{\nabla^g_j X_l + \nabla^g_l X_j}{2} \,d\mu_g.
    \end{equation}
    Since \( X \) is a conformal Killing vector field for the metric \( h \) we conclude that
    \begin{equation*}
        \frac{\nabla^g_j X_l + \nabla^g_l X_j}{2} 
        = \frac{1}{2} ( D_j X_l + D_l X_j + \Gamma_{jl}^u X_u + \Gamma_{lj}^u X_u )
        = \frac{\div_h (X)}{n}h_{jl} + \Gamma_{jl}^u g_{uv} X^v,
    \end{equation*}
    where we have now raised the index of \( X \) using \( g \). Substituting the above equality into \eqref{eq:AlmostConformal} we arrive at the formula
    \begin{equation*}
        \int_M G^g(X,\nabla^g \phi) \,d\mu_g
        = -\int_M \phi g^{ij}g^{kl}G^g_{ik} \biggl( \frac{\div_h (X)}{n}h_{jl} +  \Gamma_{jl}^u g_{uv} X^v \biggr) \,d\mu_g,
    \end{equation*}
    which is what we wanted to prove.
\end{proof}

We now show that the Ricci version of the weak ADM mass is well-defined, see also \cite[Proposition \( 4.9 \)]{GicSak}.

\begin{theorem}
    \label{thm:MassRicci}
    Let \( g \) be a \( W^{2,2}_{-(n-2)/2} \)-asymptotically Euclidean metric such that \( \Scal^g \in L^1_{-n}\) and let \( \{\chi_\alpha\}_{\alpha \geq 1} \) be a sequence of cutoff functions as in Definition \ref{def:WeakMass}. Then the Ricci version of the weak ADM mass, \( m_\textnormal{RW} \bigl( g,\Phi,\{\chi_\alpha\}_{\alpha \geq 1} \bigr) \), is well-defined and independent of the choice of cutoff functions.
\end{theorem}

\begin{proof}
    We recall that in the chart at infinity \( \Phi \colon M \setminus K \to \mathbb R^n \setminus B_R\) we have \( \Phi_* h = \delta \), that \( x^i\partial_i \) is a conformal Killing vector field on \( \mathbb R^n \setminus B_R \) and that \( \div_\delta(x^i\partial_i) = n \). Thus we can define the vector field \( X \coloneq r D r \) and note that outside the compact set \( K \), we have \( X = \Phi^* (x^i\partial_i) \) and so \( X \) is a conformal Killing vector field for \( h \) and \( \div_h (X) = n \). We also note that \( X \in L^\infty_1 \).
    
    Similarly to the proof of Theorem \ref{thm:WeakMassWellDefined}, we let \( \phi \colon M \to [0,1] \) be a locally Lipschitz function that equals \( 1 \) outside of some compact set and vanishes on \( K \). We define another sequence of functions
    \( \{\overline \chi_\alpha\}_{\alpha \geq 1} \) 
    as \( \overline \chi_\alpha \coloneq \phi \chi_\alpha \). By construction, each member of this sequence is a locally Lipschitz function with compact support.
    Using Proposition \ref{prop:ConformalKillingEinstein} with \( \overline \chi_\alpha \) in place of \( \phi \) we find
    \begin{equation}
        \label{eq:RicciWellMass1}
        \int_M G^g(X,\nabla^g \overline \chi_\alpha) \,d\mu_g
        + \int_M \overline \chi_\alpha g^{ij}g^{kl}G^g_{ik} (h_{jl} + \Gamma^u_{jl}g_{uv} X^v) \,d\mu_g = 0.
    \end{equation}
    Since \( \nabla^g \phi \) is compactly supported, there is some index \( \alpha_0 \) such that for all \( \alpha \geq \alpha_0\) we have \( \nabla^g \overline \chi_\alpha = \nabla^g \chi_\alpha + \nabla^g \phi \). Thus for \( \alpha \geq \alpha_0 \)
    \begin{equation*}
        \begin{aligned}
            G^g(X,\nabla^g \overline \chi_\alpha)
            & = G^g(X,\nabla^g \phi) + G^g(X,\nabla^g \chi_\alpha)
            \\& = G^g(X,\nabla^g \phi) + g^{ij}G^g_{ik}X^k \nabla^g_j \chi_\alpha
            \\&= G^g(X,\nabla^g \phi) + h^{ij}G^g_{ik}X^k D_j \chi_\alpha + f^{ij}G^g_{ik}X^k D_j\chi_\alpha
            \\&= G^g(X,\nabla^g \phi) + G^g(X,D\chi_\alpha) + f^{ij}G^g_{ik}X^k D_j\chi_\alpha.
        \end{aligned}
    \end{equation*}
    Combining the above equality with \eqref{eq:RicciWellMass1} we conclude that
    \begin{equation}
        \label{eq:RicciWellMass2}
        \int_M G^g(X, -D \chi_\alpha) \,d\mu_g 
        = \int_M G^g(X,\nabla^g \phi) \,d\mu_g + \mathrm I_\alpha + \mathrm{I\!I}_\alpha,
    \end{equation}
    where the integrals \( \mathrm I_\alpha \) and \( \mathrm{I\!I}_\alpha \) are given by
    \begin{equation*}
        \mathrm I_\alpha \coloneq \int_M f^{ij}G^g_{ik}X^k D_j \chi_\alpha \,d\mu_g, \quad
        \mathrm{I\!I}_\alpha \coloneq \int_M \overline \chi_\alpha g^{ij}g^{kl}G^g_{ik} (h_{jl} + \Gamma^u_{jl}g_{uv} X^v) \,d\mu_g.
    \end{equation*}
    We now consider \eqref{eq:RicciWellMass2} in the limit \( \alpha \to \infty \). First, we claim that \( \mathrm I_\alpha \) vanishes as \( \alpha \to \infty \). To show this, we start by noting the inclusion \( L^2_{-(n-2)/2} \cap L^\infty_0 \subseteq L^3_{-(n-2)/3} \), which follows from direct computation using Definition \ref{def:WeightedSobolevSpaces}. Applying Lemma \ref{lem:GeneralErrorEstimates}, we then find that \( e,f \in L^2_{-(n-2)/2} \cap L^3_{-(n-2)/3} \). It thus follows that
    \begin{equation*}
        \underbrace{f}_{L^2_{-\frac{n-2}{2}} \cap L^3_{-\frac{n-2}{3}}} \star \underbrace{X \vphantom{f}}_{L^\infty_1} 
        \in L^2_{1-\frac{n-2}{2}} \cap L^3_{1-\frac{n-2}{3}},
    \end{equation*}
    so by applying Lemma \ref{lem:TechnicalWithGagliardoNirenberg} with \( w = 1 \) and \( \tau = \frac{n-2}{2} \), we have \( f^{ij} G^g_{ik} X^k \in L^1_{1-n} \). Finally, the bound \( \sup_\alpha \lVert D\chi_\alpha \rVert_{L^\infty_{-1}} < \infty \) and the point-wise convergence \( D \chi_\alpha \to 0 \) allows us to apply Lemma \ref{lem:IntegralsToZero} to conclude that
    \begin{equation}
        \label{eq:RicciWellMass3}
        \lim_{\alpha \to \infty} \mathrm I_\alpha 
        = 0.
    \end{equation}
    In order to evaluate \( \lim_{\alpha \to \infty} \mathrm{I\!I}_\alpha \), we first use \eqref{eq:DifferenceTensorComponents} to write 
    \begin{equation*}
        h_{jl} + \Gamma^u_{jl}g_{uv} X^v 
        = (g_{jl} - e_{jl}) + \frac{X^v}{2}(D_j e_{lv} + D_l e_{jv} - D_v e_{jl}).
    \end{equation*}
    Recalling that \( G^g = \Ric^g - \frac{\Scal^gg}{2} \), the above implies
    \begin{equation*}
        \begin{aligned}
            g^{ij}g^{kl}G^g_{ik} (h_{jl} + \Gamma^u_{jl}g_{uv} X^v)
            & = g^{ij}g^{kl}G^g_{ik}g_{jl} - g^{ij}g^{kl}G^g_{ik} \biggl( e_{jl} - \frac{X^v}{2}(D_j e_{lv} + D_l e_{jv} - D_v e_{jl}) \biggr)
            \\& = -\frac{n-2}{2}\Scal^g{} - g^{-1} \star g^{-1} \star G^g \star (e + X \star De).
        \end{aligned}
    \end{equation*}
    We claim that the above is an element of \( L^1_{-n} \). To see this we note first that \( \Scal^g \in L^1_{-n} \) by assumption. Next we note that \( e \in L^2_{-(n-2)/2} \cap L^3_{-(n-2)/3}\) and since \( e \in W^{2,2}_{-(n-2)/2} \) we can argue as in the proof of Lemma \ref{lem:TechnicalWithGagliardoNirenberg} to find \( De \in L^{1,3}_{-1-(n-2)/3} \). All in all, \( g^{-1} \in L^\infty_0 \) and
    \begin{equation}
        \label{eq:DoubleSobolevInclusion}
        \underbrace{e}_{L^2_{-\frac{n-2}{2}} \cap L^3_{-\frac{n-2}{3}}} + \underbrace{X}_{L^\infty_1} \star \underbrace{De}_{L^2_{-1-\frac{n-2}{2}} \cap L^3_{-1-\frac{n-2}{3}}} \in L^2_{-\frac{n-2}{2}} \cap L^3_{-\frac{n-2}{3}}.
    \end{equation}
    Lemma \ref{lem:TechnicalWithGagliardoNirenberg} with \( w = 0 \) and \( \tau = \frac{n-2}{2} \) implies \( g^{-1} \star g^{-1} \star G^g \star (e + X \star De) \in L^1_{-n} \). This together with the uniform bound \( \sup_\alpha \lVert \overline \chi_\alpha - \phi \Vert_{L^\infty_0} \leq C \) and the point-wise convergence \( (\overline \chi_\alpha - \phi) \to 0 \) allows us to apply Lemma \ref{lem:IntegralsToZero} with \( w = 0 \) to conclude that
    \begin{equation}
        \label{eq:RicciWellMass4}
        \lim_{\alpha \to \infty} \mathrm{I\!I}_\alpha 
        = \int_M \phi g^{ij}g^{kl}G^g_{ik} (h_{jl} + \Gamma^u_{jl}g_{uv} X^v) \,d\mu_g.
    \end{equation}
    Finally, we show that \( d \mu_g \) can be replaced by \( d \mu_h \) in the integral on the left-hand side of \eqref{eq:RicciWellMass2}. Noting that \( d\mu_g = \sqrt{\det(g)/{\det(h)}} \,d\mu_h \), Lemma \ref{lem:VolumeComparison} and the bound \( \det(h) \geq C^{-1} \) gives
    \begin{equation*}
        \begin{aligned}
            \lvert \sqrt{\det(g)/{\det(h)}} - 1 \lvert
            = \frac{\lvert \sqrt{\det(g)} - \sqrt{\det(h)} \rvert}{\sqrt{\det(h)}}
            \leq 
            C \lvert \sqrt{\det(g)} - \sqrt{\det(h)} \lvert
            \leq C \lvert e \rvert_h,
        \end{aligned}
    \end{equation*}
    which implies \( \sqrt{\det(g)/{\det(h)}} - 1 \in L^2_{-(n-2)/2} \cap L^3_{-(n-2)/3} \). As \( X \in L^\infty_1 \), it follows from Lemma \ref{lem:TechnicalWithGagliardoNirenberg} with \( w = 1 \) and \( \tau = \frac{n-2}{2} \) that
    \begin{equation*}
        \underbrace{ g^{ij} \vphantom{)} }_{L^\infty_0}  
        \underbrace{ (\sqrt{\det(g)/{\det(h)}} - 1) X^k}_{L^2_{1-\frac{n-2}{2}} \cap L^3_{1-\frac{n-2}{3}} }
        G^g_{ik}
        \in L^1_{1-n}.
    \end{equation*}
    Thus a final application of Lemma \ref{lem:IntegralsToZero} with \( w = 1 \) implies
    \begin{equation}
    \label{eq:RicciWellMass5}
        \lim_{\alpha \to \infty} \int_M G^g(X,-D\chi_\alpha) \,d\mu_h 
        = \lim_{\alpha \to \infty} \int_M G^g(X,-D\chi_\alpha) \,d\mu_g.
    \end{equation}
    Combining \eqref{eq:RicciWellMass2}, \eqref{eq:RicciWellMass3}, \eqref{eq:RicciWellMass4} and \eqref{eq:RicciWellMass5}, we arrive at the equality
    \begin{equation*}
        \lim_{\alpha \to \infty} \int_M G^g(X,-D\chi_\alpha) \,d\mu_h = \int_M G^g(X,\nabla^g \phi) \,d\mu_g +
        \int_M \phi g^{ij}g^{kl}G^g_{ik} (h_{jl} + \Gamma^u_{jl}g_{uv} X^v) \,d\mu_g.
    \end{equation*}
    Recalling Definition \ref{def:WeakRicciVersions} we find
    \begin{equation}
        \label{eq:RicciMassOther}
        \begin{aligned}
            -(n-1)(n-2)\omega_{n-1} &m_\text{RW} \bigl( g,\Phi,\{\chi_\alpha\}_{\alpha \geq 1} \bigr)
            \\&= \int_M G^g(X,\nabla^g \phi) \,d\mu_g + \int_M \phi g^{ij}g^{kl}G^g_{ik} (h_{jl} + \Gamma^u_{jl}g_{uv} X^v) \,d\mu_g.
        \end{aligned}
    \end{equation}
    Just like in the proof of Theorem \ref{thm:WeakMassWellDefined}, we conclude that \( m_\text{RW}(g,\Phi,\{\chi_\alpha\}_{\alpha \geq 1}) \) is well-defined and independent of choice of cutoff functions \( \{\chi_\alpha\}_{\alpha \geq 1} \). 
\end{proof}

As the Ricci version of the weak ADM mass of \( g \) does not depend on the choice of cutoff functions we from now on use the notation \( m_\text{RW}(g,\Phi) \). We now show that the Ricci version of the weak ADM mass of \( g \) agrees with the weak ADM mass of \( g \), see also \cite[Proposition \( 4.9 \)]{GicSak}.

\begin{theorem}
    \label{thm:RicciMassIsWeakMass}
    Let \( g \) be a \( W^{2,2}_{-(n-2)/2} \)-asymptotically Euclidean metric with \( \Scal^g \in L^1_{-n} \). Then 
    \begin{equation*}
        m_\textnormal{W}(g,\Phi) 
        = m_\textnormal{RW}(g,\Phi).
    \end{equation*}
    If, in addition, \( g \) is \( W^{2,p}_{-\tau} \)-asymptotically Euclidean for some \( p > n \) and \( \tau > \frac{n-2}{2} \), then \( m_\textnormal{RW}(g,\Phi) \) is independent of the choice of chart at infinity.
\end{theorem}

\begin{proof}
    Since \( \Phi \colon (M \setminus K,h) \to (\mathbb R^n \setminus B_R,\delta) \) is an isometry, it suffices to prove the theorem in the special case when the reference manifold is \( (\mathbb R^n,\text{id}) \), the background metric structure is \( (\delta, \psi + (1-\psi) \lvert x \rvert) \) for some smooth function \( \psi \colon \mathbb R^n \to [0,1) \) which is compactly supported in \( B_1 \) and satisfies \( \psi(0) > 0 \), see Remark \ref{rem:EuclideanSobolevs}. We suppose that \( g \) is some \( {}^\delta W^{2,2}_{-(n-2)/2} \)-asymptotically Euclidean metric on \( \mathbb R^n \). 
     
    We now let \( \{\chi_\alpha\}_{\alpha \geq 1} \) be a sequence of cutoff functions on \( \mathbb{R}^n \) satisfying 
    \begin{align*}
        \chi_\alpha(x) &\equiv 1 \text{ for } \lvert x \rvert < \alpha, \\
        \chi_\alpha(x) &\equiv 0 \text{ for } \lvert x \rvert > 2\alpha,
    \end{align*}
    and
    \begin{equation*}
        \sup_\alpha \lVert D\chi_\alpha \rVert_{{}^\delta W^{1,\infty}_{-1}}
        < C.
    \end{equation*}
    We note that such a family of cutoff functions satisfies the conditions of Definition \ref{def:WeakMass}. Since \( x^i\partial_i \) is a conformal killing vector field for \( h = \delta \) on \( \mathbb R^n \), we can apply Proposition \ref{prop:ConformalKillingEinstein} with \( X \coloneq x^i\partial_i \) and \( -\chi_\alpha \) in place of \( \phi \) to find 
    \begin{equation}
        \label{eq:IntegralOnAnnulus}
        \mathrm I_\alpha 
        \coloneq \int_{\mathbb R^n} G^g(X,-\nabla^g\chi_\alpha) \,d\mu_g 
        = \int_{\mathbb R^n} \chi_\alpha g^{ij} g^{kl} G^g_{ik} (\delta_{jl} + \Gamma_{jl}^ug_{uv}X^v) \,d\mu_g 
        \eqcolon \mathrm{I\!I}_\alpha.
    \end{equation}
    We now turn to analysing the integral \( \mathrm I_\alpha \) on the left hand side above. We have
    \begin{equation}
        \label{eq:RicciLHS}
        \begin{aligned}
            \mathrm I_\alpha = 
            \begin{aligned}[t]
                \int_{\mathbb R^n} G^g(X,-D\chi_\alpha) \,d\mu_\delta
                &+ \int_{\mathbb R^n} G^g(X,-D\chi_\alpha)(\sqrt{\det(g)} - 1) \,d\mu_\delta
                \\&+ \int_{\mathbb R^n} g^{ij}f^{kl} G^g_{ik}x_j \nabla^\delta_l \chi_\alpha \sqrt{\det(g)} \,d\mu_\delta.
            \end{aligned}
        \end{aligned}
    \end{equation}
    Arguing as in the proof of Theorem \ref{thm:MassRicci} we obtain the inclusions
    \begin{equation*}
        f, (\sqrt{\det(g)} - 1) \in {}^\delta L^2_{-\frac{n-2}{2}} \cap {}^\delta L^3_{-\frac{n-2}{3}}.
    \end{equation*}
    Applying Lemma \ref{lem:TechnicalWithGagliardoNirenberg} with \( w = 1 \) and \( \tau = \frac{n-2}{2} \) twice, we obtain the inclusions
    \begin{equation*}
        G^g \star 
        \underbrace{X}_{{}^\delta L^\infty_1} \star 
        \underbrace{(\sqrt{\det(g)} - 1)}_{{}^\delta L^2_{-\frac{n-2}{2}} \cap {}^\delta L^3_{-\frac{n-2}{3}}}  
        \in L^1_{1-n}, \quad 
        G^g_{ik} \cdot
        \underbrace{f^{kl}}_{{}^\delta L^2_{-\frac{n-2}{2}} \cap {}^\delta L^3_{-\frac{n-2}{3}}} \cdot
        \underbrace{x_j \sqrt{\det(g)}}_{{}^\delta L^\infty_1}  
        \in {}^\delta L^1_{1-n}.
    \end{equation*}
    Two applications of Lemma \ref{lem:IntegralsToZero} show the last two integrals in \eqref{eq:RicciLHS} vanish as \( \alpha \to \infty \). So
    \begin{equation}
    \label{eq:LeftRicciLimit}
        \lim_{\alpha \to \infty} \mathrm I_\alpha 
        = \lim_{\alpha \to \infty} \int_{\mathbb R^n} G^g(X,-D\chi_\alpha) \,d\mu_\delta
        = -(n-1)(n-2)\omega_{n-1} m_\text{RW}(g,\Phi).
    \end{equation}
    Next we turn to analysing \( \mathrm{I\!I}_\alpha \). This integral can be rewritten as follows:
    \begin{equation}
    \label{eq:RicciRHS}
    \begin{aligned}
        \mathrm{I\!I}_\alpha =
        \begin{aligned}[t]
            \int_{\mathbb R^n} \chi_\alpha g^{ij} g^{kl} G^g_{ik}g_{jl} \,d\mu_\delta
            &+ \int_{\mathbb R^n} \chi_\alpha g^{ij} g^{kl} G^g_{ik}g_{jl}(\sqrt{\det(g)} - 1) \,d\mu_\delta 
            \\&+ \int_{\mathbb R^n} \chi_\alpha g^{ij} g^{kl} G^g_{ik} (- e_{jl} + \Gamma_{jl}^ug_{uv}X^v) \,d\mu_g.
        \end{aligned}
    \end{aligned}
    \end{equation}
    Combining Lemma \ref{lem:VolumeComparison} and Lemma \ref{lem:TechnicalWithGagliardoNirenberg} with \( w = 0 \) and \( \tau = \frac{n-2}{2} \) we conclude that
    \begin{equation*}
        \underbrace{g^{ij} g^{kl}}_{{}^\delta L^\infty_0} G^g_{ik}\underbrace{g_{jl}}_{{}^\delta L^\infty_0} \underbrace{(\sqrt{\det(g)} - 1)}_{{}^\delta L^2_{-\frac{n-2}{2}} \cap {}^\delta L^3_{-\frac{n-2}{3}}} \in L^1_{-n}.
    \end{equation*}
    Recalling \eqref{eq:DoubleSobolevInclusion}, we can once again apply Lemma \ref{lem:TechnicalWithGagliardoNirenberg} with \( w = 0 \) and \( \tau = \frac{n-2}{2} \) to find
    \begin{equation*}
        \underbrace{g^{ij} g^{kl}}_{{}^\delta L^\infty_0} G^g_{ik} (\underbrace{-e_{jl} + \Gamma_{jl}^ug_{uv}X^v}_{{}^\delta L^2_{-\frac{n-2}{2}} \cap {}^\delta L^3_{-\frac{n-2}{3}}}) \underbrace{\sqrt{\det(g)}}_{{}^\delta L^\infty_0}
        \in L^1_{-n}. 
    \end{equation*}
    Thus two applications of Lemma \ref{lem:IntegralsToZero} show that the last two integrals in \eqref{eq:RicciRHS} vanish as \( \alpha \to \infty \), hence
    \begin{equation}
    \label{eq:RightRicciLimit}
        \lim_{\alpha \to \infty} \mathrm{I\!I}_\alpha
        = \lim_{\alpha \to \infty} \int_{\mathbb R^n} \chi_\alpha g^{ij}g^{kl} G^g_{ik}g_{jl} \,d\mu_\delta
        = -\frac{n-2}{2}\lim_{\alpha \to \infty} \int_{\mathbb R^n} \chi_\alpha\Scal^g \,d\mu_\delta.
    \end{equation}
    Combining \eqref{eq:IntegralOnAnnulus}, \eqref{eq:LeftRicciLimit} and \eqref{eq:RightRicciLimit}, as well as applying Proposition \ref{prop:FirstOrderScalar}, we now find
    \begin{equation}
        \label{eq:Penultimate}
        2(n-1)\omega_{n-1} m_\text{RW}(g,\Phi) 
        = \lim_{\alpha \to \infty}\int_{\mathbb R^n} \chi_\alpha D_i\bigl( g^{kl}g^{ij}(D_ke_{jl} - D_je_{kl}) \bigr) \,d\mu_\delta 
        + \int_{\mathbb R^n} \chi_\alpha\mathcal Q^S \,d\mu_\delta.
    \end{equation}
    where $\mathcal Q^S$ is of the form $g^{-1} \star g^{-1} \star g^{-1} \star De \star De \in L^1_{-n}$. The point-wise convergence \( \chi_\alpha \to 0 \) combined with \( \mathcal Q^S \in L^1_{-n} \) and \( \sup_\alpha \lVert \chi_\alpha\rVert_{L^\infty_0} < \infty \) allows us to apply Lemma \ref{lem:IntegralsToZero} to find
    \begin{equation*}
        \lim_{\alpha \to \infty}\int_{\mathbb R^n} \chi_\alpha\mathcal Q^S \,d\mu_\delta 
        = 0. 
    \end{equation*}
    An application of the divergence theorem to the first term in the right hand side of \eqref{eq:Penultimate}, which is justified since each $\chi_\alpha$ has compact support, yields
    \begin{equation*}
        m_\text{RW}(g,\Phi)
        = \frac{1}{2(n-1)\omega_{n-1}}\lim_{\alpha \to \infty}\int_{\mathbb R^n} (-D_i\chi_\alpha)g^{kl}g^{ij}(D_ke_{jl} - D_je_{kl}) \,d\mu_\delta 
        = m_\text{W}(g,\Phi),
    \end{equation*}
    where the second equality is a consequence of Theorem \ref{thm:WeakMassWellDefined}.

    When \( g \) is \( W^{2,p}_{-\tau} \)-asymptotically Euclidean for \( p > n\) and \( \tau > \frac{n-2}{2} \), then it is also \( W^{1,p}_{-\tau} \)-asymptotically Euclidean. Theorem \ref{thm:CoordinateInvariance} then implies that \( m_\text{W}(g,\Phi) \) is independent of the choice of chart at infinity and hence also \( m_\text{RW}(g,\Phi) \). 
\end{proof}
Finally we show that when \( g \) has the regularity needed in Definition \ref{def:WeakRicciVersions}, then the Ricci version of the weak ADM mass and the Ricci version of the ADM mass of \( (M,g) \) agree. 

\begin{theorem}
    \label{thm:WeakMassAndWeakRicciMassEqual}
    Let \( g \) be a \( C^2 \)-asymptotically Euclidean metric of order \( \tau > \frac{n-2}{2} \) as in Definition \ref{def:SmoothAE} with \( \Scal^g \in L^1_{-n} \) and \( g \in C^3_\textnormal{loc}(M) \). Then \( m_\textnormal{R}(g,\Phi), m_\textnormal{ADM}(g,\Phi), m_\textnormal{W}(g,\Phi) \) and \( m_\textnormal{RW}(g,\Phi) \) are all independent of the choice of chart at infinity. Moreover we have the equalities
    \begin{equation*}
        m_\textnormal{R}(g) 
        = m_\textnormal{ADM}(g) 
        = m_\textnormal{W}(g)
        = m_\textnormal{RW}(g).
    \end{equation*}
\end{theorem}

\begin{proof}
    The first equality follows from \cite[Theorem \( 2.3 \)]{Herz}, the second from Theorem \ref{thm:C2isSobolevMass} and the last from Theorem \ref{thm:RicciMassIsWeakMass}. By Theorem \ref{thm:CoordinateInvariance}, \( m_\text{W}(g,\Phi) \) is independent of the choice of chart at infinity and hence so are \( m_\text{R}(g,\Phi) \), \( m_\text{ADM}(g,\Phi) \) and \( m_\text{RW}(g,\Phi) \) as well.
\end{proof}